\documentclass[11pt]{amsart}   % includes 11pt font size (default 10pt)

\usepackage{verbatim}
\usepackage{amssymb,amsthm,amsmath}
\usepackage{color}
\definecolor{darkblue}{rgb}{0, 0, .4}
\definecolor{grey}{rgb}{.7, .7, .7}

\usepackage{ifpdf}
\ifpdf
  \usepackage[pdftex]{epsfig}

  \usepackage{lscape}

  \usepackage[breaklinks]{hyperref}
  \hypersetup{
            colorlinks=true,
            linkcolor=darkblue,
            anchorcolor=darkblue,
            citecolor=darkblue,
            pagecolor=darkblue,
            urlcolor=darkblue,
            pdftitle={},
            pdfauthor={}
  }
\else
  \usepackage[dvips]{epsfig}
  \usepackage{lscape}

  \newcommand{\href}[2]{#2}
  \newcommand{\url}[2]{#2}
\fi

\newtheorem{theorem}{Theorem}[section]
\newtheorem{lemma}[theorem]{Lemma}

\theoremstyle{definition}
\newtheorem{definition}[theorem]{Definition}
\newtheorem{example}[theorem]{Example}

\newtheorem{openproblem}[theorem]{Open Problem}
\newtheorem{question}[theorem]{Question}

\newtheorem*{remark}{Remark}

\numberwithin{equation}{section}

\theoremstyle{theorem}
\newtheorem{corollary}[theorem]{Corollary}

\newtheorem{proposition}[theorem]{Proposition}
\newtheorem{conjecture}[theorem]{Conjecture}

\newcommand{\N}[0]{\mathbb{N}}
\newcommand{\Z}[0]{\mathbb{Z}}

\newcommand{\R}[0]{\mathbb{R}}

\newcommand{\Cat}[0]{\mathsf{Cat}}

\newcommand{\row}{\rightarrow}

\usepackage{graphicx}

\newcommand{\lf}{\lfloor}
\newcommand{\rf}{\rfloor}
\newcommand{\tld}{\widetilde}
\newcommand{\lam}{\lambda}

\newcommand{\calW}{\mathcal{W}} % Window notation
\newcommand{\calL}{\mathcal{L}} % Lattice path

\newcommand{\Shi}{\ensuremath{\mathsf{Shi}}}
\newcommand{\stat}{\ensuremath{\mathsf{stat}}}
\newcommand{\maj}{\ensuremath{\mathsf{maj}}}
\newcommand{\qbinom}[2]{\genfrac{[}{]}{0pt}{}{#1}{#2}_q}
\newcommand{\qqbinom}[2]{\genfrac{[}{]}{0pt}{}{#1}{#2}_{q^2}}

\renewcommand{\mod}{\ \mathrm{mod}\ }

\newcommand{\A}{A_{\circ}}
\newcommand{\e}{\varepsilon}

\usepackage{tikz}

\begin{document}

\title{Results and conjectures on simultaneous core partitions}

\author{Drew Armstrong}
\address{Department of Mathematics \\ University of Miami \\ Coral Gables, FL 33146}
\email{\href{mailto:armstrong@math.miami.edu}{\texttt{armstrong@math.miami.edu}}}
\urladdr{\url{http://www.math.miami.edu/\~armstrong/}}

\author{Christopher R.\ H.\ Hanusa}
\address{Department of Mathematics \\ Queens College (CUNY) \\ 65-30 Kissena Blvd. \\ Flushing, NY 11367}
\email{\href{mailto:chanusa@qc.cuny.edu}{\texttt{chanusa@qc.cuny.edu}}}
\urladdr{\url{http://qcpages.qc.edu/~chanusa/}}

\author{Brant C. Jones}
\address{Department of Mathematics and Statistics \\ James Madison University \\ Harrisonburg, VA 22807}
\email{\href{mailto:brant@math.jmu.edu}{\texttt{brant@math.jmu.edu}}}
\urladdr{\url{http://educ.jmu.edu/\~jones3bc/}}


%\comment{05A17 Partitions of integers, 05E15 Combinatorial problems concerning the classical groups,  05A30 $q$-calculus and related topics, 05A10 Tableaux, etc. 05A15 Exact enumeration problems, generating functions, 20F55 Reflection and Coxeter groups}}

\date{\today}

%%%%%%%%%%%%%%%%%%%%%%%%%%%%%%%%%%%%%%%%%%%%%%%%%%%%%%%%%%%%%%%%%%%%%
%  Abstract
%%%%%%%%%%%%%%%%%%%%%%%%%%%%%%%%%%%%%%%%%%%%%%%%%%%%%%%%%%%%%%%%%%%%%

\begin{abstract}
An $n$-core partition is an integer partition whose Young diagram contains no hook lengths equal to $n$.  We consider partitions that are simultaneously $a$-core and $b$-core for two relatively prime integers $a$ and $b$. These are related to abacus diagrams and the combinatorics of the affine symmetric group (type $A$).   We observe that self-conjugate simultaneous core partitions correspond to the combinatorics of type $C$, and use abacus diagrams to unite the discussion of these two sets of objects.

In particular, we prove that $2n$- and $(2mn+1)$-core partitions correspond naturally to dominant alcoves in the $m$-Shi arrangement of type~$C_n$, generalizing a result of Fishel--Vazirani for type~$A$.  We also introduce a major index statistic on simultaneous $n$- and $(n+1)$-core partitions and on self-conjugate simultaneous $2n$- and $(2n+1)$-core partitions that yield $q$-analogues of the Coxeter-Catalan numbers of type~$A$ and type~$C$.

We present related conjectures and open questions on the average size of a simultaneous core partition, $q$-analogs of generalized Catalan numbers, and generalizations to other Coxeter groups.  We also discuss connections with the cyclic sieving phenomenon and $q,t$-Catalan numbers.
\end{abstract}

\maketitle

%%%%%%%%%%%%%%%%%%%%%%%%%%%%%%%%%%%%%%%%%%%%%%%%%%%%%%%%%%%%%%%%%%%%%
%  Section
%%%%%%%%%%%%%%%%%%%%%%%%%%%%%%%%%%%%%%%%%%%%%%%%%%%%%%%%%%%%%%%%%%%%%
\section{Introduction}\label{sec:intro}

Let $a$ and $b$ be coprime positive integers. In this paper we will examine integer partitions %$\lambda\vdash n$ (for any $n$) 
that are simultaneously $a$-core and $b$-core (i.e. have no cells of hook length $a$ or $b$). Jaclyn Anderson proved that the number of such ``$(a,b)$-cores" is finite and has a nice closed formula generalizing the Catalan numbers. Moreover, her proof was elegant and bijective. Since then, it has become obvious that simultaneous core partitions are at the intersection of some very interesting algebra, combinatorics, and geometry. We will state some intriguing conjectures about $(a,b)$-cores, we will generalize work of Fishel and Vazirani on Shi arrangements from type $A$ to type $C$, and we will investigate a ``major index"-type statistic on simultaneous cores. Please enjoy.

The paper is organized as follows. In Section \ref{sec:co} we introduce definitions and background material on core partitions. We also state some conjectures and open problems that motivate the rest of the paper. Section~\ref{sec:abacus} introduces precise definitions of abacus diagrams, which serve as the basis for the proofs of our results.  The focus of Section~\ref{sec:geometry} is alcoves in $m$-Shi arrangements of types $A$ and $C$.  The key result is Theorem~\ref{thm:alcove}, which characterizes $m$-minimal and $m$-bounded regions as a simultaneous core condition, generalizing the result of Fishel and Vazirani \cite{FV1} through a unified method.  Theorem~\ref{t:qA} gives a major index statistic on simultaneous core partitions to find a $q$-analog of the Catalan numbers of type $A$ and type $C$ using abacus diagrams and their bijection with lattice paths; this is the main goal of Section~\ref{sec:maj}.  We conclude in Section~\ref{sec:fq} with a few more open problems motivated by this paper.

%%%%%%%%%%%%%%%%%%%%%%%%%%%%%%%%%%%%%%%%%%%%%%%%%%%%%%%%%%%%%%%%%%%%%
%  Section
%%%%%%%%%%%%%%%%%%%%%%%%%%%%%%%%%%%%%%%%%%%%%%%%%%%%%%%%%%%%%%%%%%%%%
\section{Background and Conjectures}\label{sec:co}

A {\bf partition} of the integer $n\in\N$ is an unordered multiset of positive integers $\lambda_1\geq \lambda_2\geq \cdots \geq \lambda_k>0$ such that $\sum_{i=1}^k \lambda_i=n$. We will write this as $\lambda=(\lambda_1,\lambda_2,\ldots,\lambda_k)\vdash n$, and say that the {\bf size} of the partition is $n$ and the {\bf length} of the partition is $k$. We will often associate a partition $\lambda$ with its {\bf Young diagram}, which is an array of boxes aligned up and to the left, placing $\lambda_i$ boxes in the $i$-th row from the top. For example, Figure \ref{fig:young_diagram} shows the Young diagram for the partition $(5,4,2,1,1)\vdash 13$.
\begin{figure}[htb]
\begin{center}
\includegraphics[scale=1]{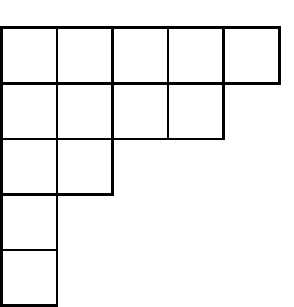}
\end{center}
\label{fig:young_diagram}
\caption{A Young diagram}
\end{figure}

To each box $B\in\lambda$ in the partition we associate its {\bf hook length} $h(B)$, which is the number of cells directly below and directly to the right of $B$ (including $B$ itself). For example, in Figure \ref{fig:58core} we have labeled each box with its hook length. An example hook of length $6$ has been shaded.

\begin{figure}[htb]
\begin{center}
\includegraphics[scale=1]{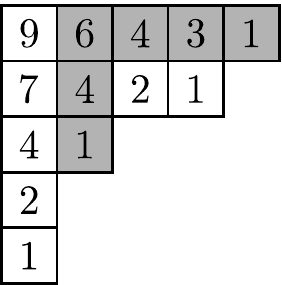}
\end{center}
\caption{A $(5,8)$-core}
\label{fig:58core}
\end{figure}

We say that an integer partiton $\lambda\vdash n$ is {\bf $a$-core} if it has no boxes of hook length $a$. The reason for the name ``$a$-core" is as follows. If a box $B\in\lambda$ has hook length $h(B)=a$, then we could also say that $\lambda$ has a {\bf rim $a$-hook} consisting of the cells along the boundary of $\lambda$, traveling between the box furthest below $B$ and the box furthest to the right of $B$. The boxes of this rim hook can then be stripped away to create a smaller partition. For example, in Figure \ref{fig:58core_strip} we have stripped away the rim $6$-hook from our previous example.

\begin{figure}[htb]
\begin{center}
\includegraphics[scale=1]{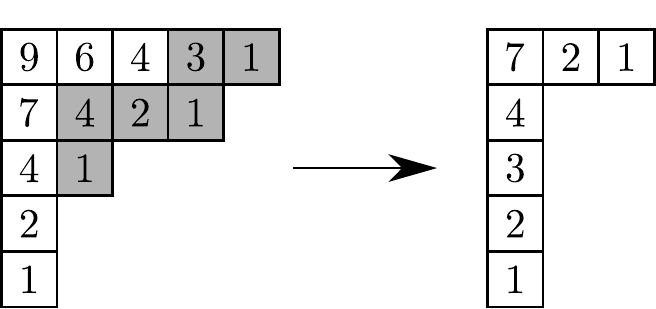}
\end{center}
\caption{Removing a rim $6$-hook}
\label{fig:58core_strip}
\end{figure}

If we continue stripping away rim $a$-hooks from $\lambda$ we eventually arrive at a partition $\tilde{\lambda}$ that has no boxes of hook length $a$. The resulting $\tilde{\lambda}$ is called {\em the} $a$-core of $\lambda$. For example, we see above that $(3,1,1,1,1)\vdash 7$ is the $6$-core of $(5,4,2,1,1)\vdash 13$. Thus it makes sense to say that a partition $\lambda$ is ``$a$-core" when it is equal to its own $a$-core.

However, it is not obvious from the above construction that the $a$-core of a partition is well-defined. We must show that the resulting partition is independent of the order in which we remove rim $a$-hooks. This at first seems difficult, but there is a beautiful argument of James and Kerber  \cite[Lemma 2.7.13]{JK} that makes it easy. In order to explain this we introduce the {\bf abacus} notation for integer partitions.

First note that we can encode an integer partition as an infinite binary string beginning with $0$s and ending with $1$s. To do this we think of the partition sitting in an infinite corner. Then we replace vertical steps by $0$s and horizontal steps by $1$s. For example, our favorite partition $(5,4,2,1,1)\vdash 13$ yields the string $\cdots 00100101101011\cdots$.
This {\bf boundary string} contains useful information. For example, the boxes of the partition are in bijection with {\bf inversions} in the string, i.e., pairs of symbols in which $1$ appears to the left of $0$. Furthermore, boxes with hook length $a$ correspond to inversions of ``length $a$" (i.e., with the $1$ and $0$ separated by $a-1$ intervening symbols). Using this language we see that the removal of a rim $a$-hook corresponds to converting an inversion of length $a$ into a non-inversion of length $a$. For example, in Figure \ref{fig:boundary_word} we have replaced the substring $1011010$ by $0011011$.

\begin{figure}[htb]
\begin{center}
\includegraphics[scale=1]{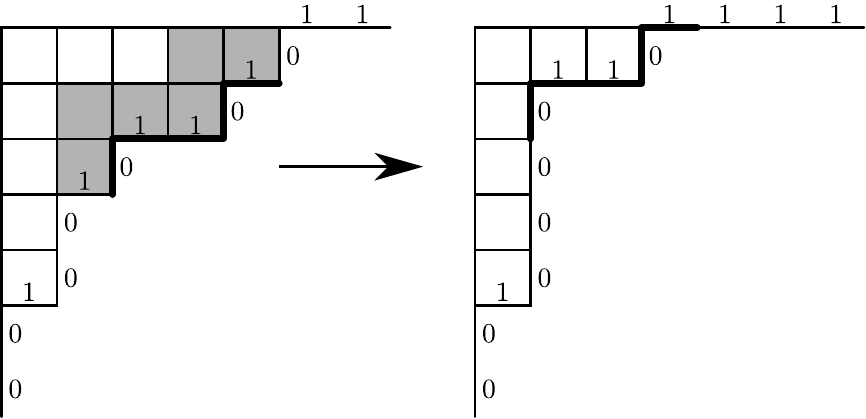}
\end{center}
\caption{Removing a rim hook $=$ removing an inversion}
\label{fig:boundary_word}
\end{figure}

Finally, we can wind the boundary string around a cycle of length $a$ to obtain an {\bf abacus diagram}. Here we read the boundary string from left to right and then proceed to the next row below. We think of the columns as {\bf runners}, the $0$s as {\bf beads}, and the $1$s as {\bf gaps}. In this language, the removal of a rim $a$-hook corresponds to {\em sliding a bead up one level into a gap}, as shown in Figure \ref{fig:abacus_diagram}.

\begin{figure}[htb]
\begin{center}
\includegraphics[scale=1]{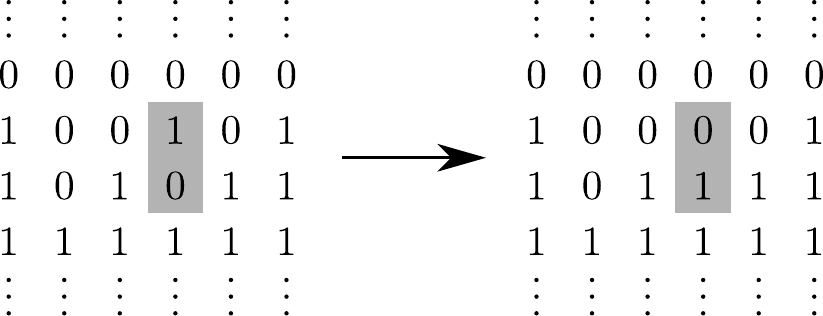}
\end{center}
\caption{Removing an inversion $=$ sliding a bead}
\label{fig:abacus_diagram}
\end{figure}

We see from this that $a$-cores are well-defined: we simply push all the beads up on their runners until there are no more gaps and we say that the abacus diagram is now {\bf $a$-flush}.

Now we turn to the main subject of the current paper: {\bf simultaneous} core partitions. We say that an integer partition $\lambda\vdash n$ is {\bf $(a,b)$-core} if it is simultaneously $a$-core and $b$-core. Our primary interest in $(a,b)$-cores is motivated by the following result of Jaclyn Anderson from 2002.

\begin{theorem}\cite{Anderson}
The total number of $(a,b)$-core partitions is finite if and only if $a$ and $b$ are coprime, in which case the number is
\begin{equation}\label{eq:catalan}
\frac{1}{a+b}\binom{a+b}{a,b}=\frac{(a+b-1)!}{a!\,b!}.
\end{equation}
\end{theorem}
Note: when $a$ and $b$ are {\em not} coprime, Formula~\eqref{eq:catalan} is not necessarily even an integer.  We now sketch the idea of Anderson's proof of Proposition \ref{p:Anderson}.  Note that an integer partition $\lambda$ is completely determined by the hook lengths of the boxes in its first column.  These boxes are also in bijection with a set of beads in a certain normalized abacus diagram.  If $\lambda$ is $(a,b)$-core, then its beads must be flush simultaneously in {\em two} ways.  Anderson's construction beautifully creates a shifted abacus diagram where the $a$-flush condition is horizontal and the $b$-flush condition is vertical.

The correspondence between hook lengths in the first column and the beads of the flush abacus defines a bijection between $(a,b)$-cores and lattice paths in $\Z^2$ from $(0,0)$ to $(b,a)$, staying above the diagonal. We call these {\bf $(a,b)$-Dyck paths}. For example, Figure \ref{fig:anderson4} shows Anderson's bijection applied to our favorite partition.

\begin{figure}[htb]
\begin{center}
\includegraphics[scale=1]{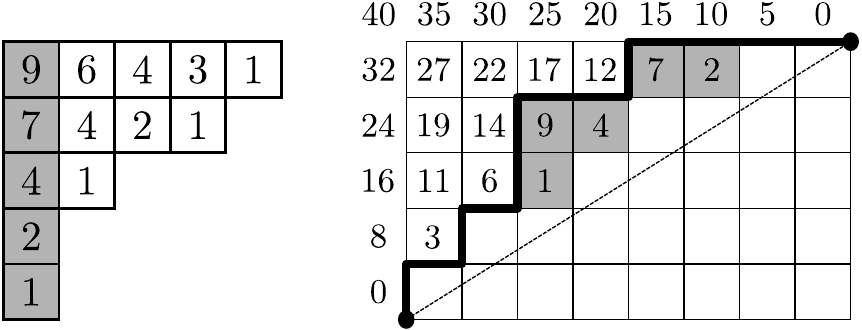}
\end{center}
\caption{Anderson's bijection}
\label{fig:anderson4}
\end{figure}

It was known since at least Bizley \cite{Bizley} that the $(a,b)$-Dyck paths (with $a$ and $b$ coprime) are counted by Formula~\eqref{eq:catalan}, which is a generalization of the classical Catalan numbers.

Next we discuss {\em self-conjugate} core partitions. Given an integer partition $\lambda=(\lambda_1,\lambda_2,\ldots,\lambda_k)\vdash n$, we define the {\bf conjugate} partition $\lambda'=(\lambda'_1,\lambda'_2,\ldots,\lambda'_\ell)\vdash n$ by setting
\begin{equation*}
\lambda'_j:=\#\{j: \lambda_j\geq i\}.
\end{equation*}
Equivalently, the Young diagram of $\lambda'$ is obtained by reflecting the Young diagram of $\lambda$ across the main diagonal. For example, Figure \ref{fig:58core_conjugate} shows that the partitions $(5,4,2,1,1)\vdash 13$ and $(5,3,2,2,1)\vdash 13$ are conjugate.

\begin{figure}[htb]
\begin{center}
\includegraphics[scale=1]{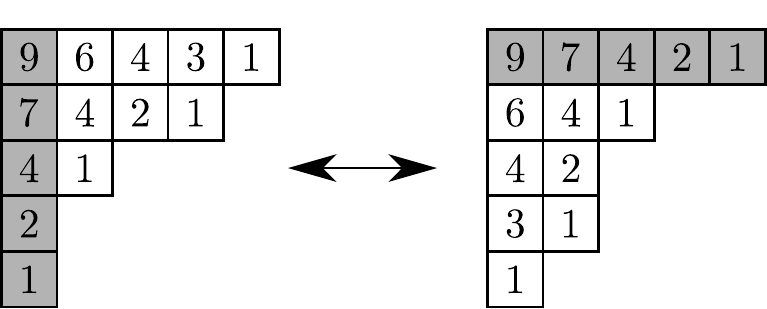}
\end{center}
\caption{Conjugation of a core partition}
\label{fig:58core_conjugate}
\end{figure}

Observe that $\lambda$ is $(a,b)$-core if and only if $\lambda'$ is $(a,b)$-core. Thus one may be interested in studying the {\bf self-conjugate} $(a,b)$-cores. In 2009, Ford, Mai and Sze proved the following analogue of Anderson's theorem.

\begin{theorem}\cite{FMS}
If $a$ and $b$ are coprime, then the number of self-conjugate $(a,b)$-cores is
\begin{equation*}
\binom{\big\lfloor\frac{a}{2}\big\rfloor+\big\lfloor\frac{b}{2}\big\rfloor}{\big\lfloor\frac{a}{2}\big\rfloor,\big\lfloor\frac{b}{2}\big\rfloor}=\frac{(\big\lfloor\frac{a}{2}\big\rfloor+\big\lfloor\frac{b}{2}\big\rfloor)!}{\big\lfloor\frac{a}{2}\big\rfloor !\, \big\lfloor \frac{b}{2}\big\rfloor !}
\end{equation*}
\end{theorem}

Before discussing their proof, we  make some numerological observations. Given $k\leq n\in\N$, we define the standard {\bf $q$-integer}, {\bf $q$-factorial}, and {\bf $q$-binomial coefficient}:
\begin{align*}
[n]_q &:= 1+q+q^2+\cdots +q^{n-1},\\
[n]_q! &:= [n]_q[n-1]_q\cdots [2]_q[1]_q,\\
{n \brack k}_q &:=\frac{[n]_q!}{[k]_q!\, [n-k]_q!}.
\end{align*}
Inspired by Formula \eqref{eq:catalan} we define the {\bf rational $q$-Catalan number}:
\begin{equation}
\Cat_q(a,b):=\frac{1}{[a+b]_q}{a+b \brack a,b}_q = \frac{[a+b-1]_q!}{[a]_q!\, [b]_q!}.
\end{equation}
Observe that the case $(a,b)=(n,n+1)$ corresponds to the classical ``$q$-Catalan number" of MacMahon:
\begin{equation*}
\Cat_q(n,n+1)=\frac{1}{[n+1]_q}{2n \brack n}_q.
\end{equation*}
MacMahon proved that $\Cat_q(n,n+1)$ is in $\N[q]$ by defining a statistic on lattice paths, which is now called ``major index" (in honor of the fact that MacMahon held the rank of major in the British Army). First, note that the final step of an $(n,n+1)$-Dyck path must be horizontal, and removing this last step defines a bijection between $(n,n+1)$-Dyck paths and {\bf classical Dyck paths} from $(0,0)$ to $(n,n)$, staying {\em weakly} above the diagonal. Given a path $P$ that starts at $(0,0)$ and takes right steps and up steps---this includes classical Dyck paths---we define the {\bf major index} as follows. Begin at $(0,0)$ and call this vertex $0$. Then $\maj(P)$ is the sum of $i$ such that the step $(i-1)\to i$ is horizontal and $i\to (i+1)$ is vertical (or in other words, the $i$-th vertex is a ``valley" of the path). MacMahon proved the following.

\begin{theorem}\cite[page 214]{MacMahon}
We have
\begin{equation*}
\sum_P q^{\maj(P)} = \frac{1}{[n+1]_q}{2n \brack n}_q,
\end{equation*}
where the sum is over classical Dyck paths $P$ (equivalently, $(n,n+1)$-Dyck paths $P$).
\end{theorem}

We note that the ``$q$-Catalan" numbers $\Cat_q(n,n+1)$ were studied by Furlinger and Hofbauer \cite{FurlingerHofbauer}. The following problem\footnote{Dennis Stanton, personal communication.} is open.

\begin{openproblem}\label{p:qcat}
Given $a,b\in\N$ coprime, define a statistic $\stat$ on $(a,b)$-Dyck paths, or equivalently on $(a,b)$-cores, such that
\begin{equation*}
\sum_P q^{\stat(P)} = \Cat_q(a,b) = \frac{1}{[a+b]_q}{a+b \brack a,b}_q. 
\end{equation*}
Preferably we would have $\stat=\maj$ when $(a,b)=(n,n+1)$.
\end{openproblem}

We would even like an elementary proof that $\Cat_q(a,b)$ is a polynomial in $\N[q]$ when $a,b\in\N$ are coprime. (The only known proof of this fact \cite[Section 1.12]{GordonGriffeth} uses the representation theory of rational Cherednik algebras.) We  conjecture a solution to Problem \ref{p:qcat} in Conjecture \ref{c:maj} below.

Here is another open problem.

\begin{openproblem}\label{p:sieving}
One can verify the following evaluation at $q=-1$:
\begin{equation*}
\left.\frac{1}{[a+b]_q}{a+b \brack a,b}_q\right|_{q=-1}=\binom{\big\lfloor\frac{a}{2}\big\rfloor+\big\lfloor\frac{b}{2}\big\rfloor}{\big\lfloor\frac{a}{2}\big\rfloor,\big\lfloor\frac{b}{2}\big\rfloor}.
\end{equation*}
Is this an example of a ``cyclic sieving phenomenon''? (See the survey \cite{Sagan} for details.) That is, does there exist a cyclic group action on $(a,b)$-cores such that ``rotation by $180^\circ$" corresponds to conjugation of the partition? Can one use this to view the results of Anderson and Ford-Mai-Sze as two special cases of a more general theorem?
\end{openproblem}

We  also state a related conjecture.

\begin{conjecture}
Let $a,b\in\N$ be coprime. Then the average size of an $(a,b)$-core and the average size of a self-conjugate $(a,b)$-core are both equal to
\begin{equation}\label{eq:avg}
\frac{(a+b+1)(a-1)(b-1)}{24}.
\end{equation}
Olsson and Stanton \cite{OlssonStanton} proved that there is a unique $(a,b)$-core of maximum size (which happens to be self-conjugate), and this size is
\begin{equation*}
\frac{(a^2-1)(b^2-1)}{24}.
\end{equation*}
Thus we can rephrase our conjecture by stating that the average ratio between an $(a,b)$-core and the {\em largest} $(a,b)$-core is
\begin{equation*}
\frac{(a+b+1)}{(a+1)(b+1)}.
\end{equation*}
\end{conjecture}

This conjecture was observed experimentally by the first author in 2011 and has been publicized informally since then.  While the second author has been able to prove Formula~\eqref{eq:avg} for small values of $a$, a framework to prove the conjecture in general has been met with frustratingly little progress. The fact that the same average size holds for both $(a,b)$-cores and self-conjugate $(a,b)$-cores makes it seem that the conjecture may be related to Problem \ref{p:sieving}. (Update: The case $(a,b)=(n,n+1)$ of the conjecture has recently been proved by Stanley and Zanello \cite{stanley-zanello}.)

Now we return to a discussion of the Ford-Mai-Sze theorem. Their proof is an ingenious bijection between self-conjugate $(a,b)$-cores and lattice paths in a $\big\lfloor\frac{a}{2}\big\rfloor\times \big\lfloor\frac{b}{2}\big\rfloor$ rectangle. But perhaps too ingenious, since it does not involve abacus diagrams, and it makes no connection to Anderson's theorem. In the current paper we  give a more natural interpretation of the Ford-Mai-Sze result using the type-$C$ abacus model recently developed by the second and third authors \cite{HJ12}. In this language we will see that there is a natural framework in which Anderson's result corresponds to the affine Weyl group of type $A$ and the Ford-Mai-Sze result corresponds to the affine Weyl group of type $C$.

We will also describe an analogous relationship between the {\bf Shi hyperplane arrangements} of types $A$ and $C$. Given a finite cystallographic root system $\Phi\subseteq V$ with positive roots $\Phi^+$, the $m$-Shi arrangement consists of the hyperplanes
\begin{equation*}
\Shi^m(\Phi) = \cup_{\alpha\in\Phi^+} \left\{ H_{\alpha,-m+1}, H_{\alpha,-m+2},\cdots , H_{\alpha,m-1}, H_{\alpha,m}\right\},
\end{equation*}
where $H_{\alpha,k}=\left\{ x\in V: \langle x,\alpha\rangle =k\right\}$. Each chamber of the $m$-Shi arrangement is a union of alcoves, which can be thought of as elements of the corresponding affine Weyl group. These alcoves, in turn, can be encoded by abacus diagrams.

Fishel and Vazirani considered the $m$-Shi arrangement of type $A_{n-1}$. They used abacus diagrams to construct a bijection between {\em minimal} alcoves in all the chambers of $\Shi^m(A_{n-1})$ and $(n,mn+1)$-core partitions \cite{FV1}, and a bijection between {\em maximal} alcoves in the {\em bounded} chambers of $\Shi^m(A_{n-1})$ and $(n,mn-1)$-core partitions \cite{FV2}. In Section \ref{sec:geometry} of this paper we use the type-$C$ abacus diagrams of Hanusa and Jones to prove the following new result: There is a bijection between {\em minimal} alcoves of the chambers of $\Shi^m(C_n)$ and self-conjugate $(2n,2mn+1)$-cores, and there is a bijection between {\em maximal} alcoves in the {\em bounded} chambers of $\Shi^m(C_n)$ and self-conjugate $(2n,2mn-1)$-cores. It is an open problem to extend the theory to types $B$ and $D$. The study of $(a,b)$-cores when $b=\pm 1 \mod{a}$ is called the ``Fuss-Catalan" level of generality. It is also an open problem to extend these results on Shi arrangements to more general $b$.

Finally, we return to the problem of $q$-Catalan numbers. For any finite reflection group $G$ with ``degrees" $d_1\leq d_2\leq \cdots \leq d_\ell=:h$, one can define a {\bf $q$-Catalan number}
\begin{equation*}
\Cat_q(G)=\prod_{i=1}^\ell \frac{[h+d_i]_q}{[d_i]_q}.
\end{equation*}
For definitions see Section 2.7 of \cite{armstrong}. It is known that $\Cat_q(G)\in\N[q]$ (see \cite{GordonGriffeth}), but combinatorial interpretations of this fact are missing in almost all cases. The (symmetric) group of type $A_{n-1}$ has degrees $2,3,\ldots,n$, and so
\begin{equation*}
\Cat_q(A_{n-1})=\frac{1}{[n+1]_q}{2n \brack n}_q,
\end{equation*}
which we discussed above. The group of type $C_n$ (the hyperoctahedral group) has degrees $2,4,6,\ldots,2n$, and so
\begin{equation*}
\Cat_q(C_n)={2n \brack n}_{q^2}.
\end{equation*}
In Section \ref{sec:maj} of this paper we will describe explicit ``major index"-type statistics on $(n,n+1)$-cores and self-conjugate $(2n,2n+1)$-cores which explain the numbers $\Cat_q(A_{n-1})$ and $\Cat_q(C_n)$. These are obtained by transfering the standard major index on lattice paths via the bijections of Anderson and Ford-Mai-Sze. The new observation is that the statistics are natural to express in the language of abacus diagrams. It is an open problem to define a similar statistic on general $(a,b)$-cores (see Problem \ref{p:qcat}). However, we will now state a conjecture that may solve the problem.

Given an $(a,b)$-core $\lambda$, we say that its {\bf $b$-boundary} consists of the skew-subdiagram of boxes with hook lengths $<b$. We define the {\bf $a$-rows} of $\lambda$ as follows. Consider the boxes in the first column of $\lambda$ and reduce their hook lengths modulo $a$. Consider the highest row in each residue class. These are the $a$-rows of the diagram.

\begin{definition}
Let $\lambda$ be an $(a,b)$-core partition with $a<b$ coprime. The {\bf skew length} $s\ell(\lambda)$ is the number of boxes of $\lambda$ that are simultaneously in the $b$-boundary and the $a$-rows of $\lambda$.
\end{definition}

For example, the partition $(7,6,2,2,2,2)\vdash 21$ is a $(7,8)$-core, and Figure \ref{fig:78core_maj} shows that its skew length is $13$.

\begin{figure}[htb]
\begin{center}
\includegraphics[scale=1]{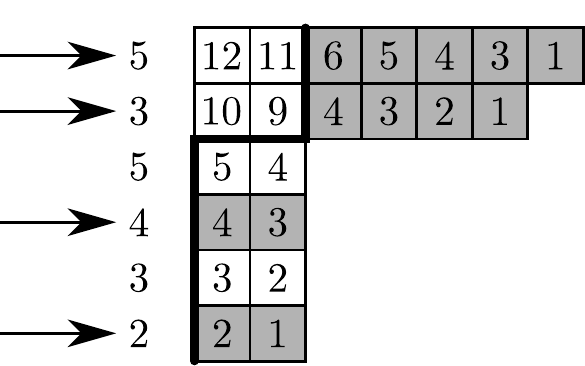}
\end{center}
\caption{The $7$-rows of the $8$-boundary of a $(7,8)$-core}
\label{fig:78core_maj}
\end{figure}

Recall that the {\bf length} of an integer partition $\ell(\lambda)$ is its number of nonzero rows. We make the following conjecture.

\begin{conjecture}\label{c:maj}
Let $a<b$ be coprime. Then we have
\begin{equation*}
\sum_{\lambda} q^{\ell(\lambda)+s\ell(\lambda)} = \frac{1}{[a+b]_q}{a+b \brack a,b}_q,
\end{equation*}
where the sum is over $(a,b)$-cores $\lambda$.
\end{conjecture}

For example, the $(7,8)$-core shown above has $\ell(\lambda)+s\ell(\lambda)=6+13=19$. We might be tempted to define a statistic $\maj(\lambda)=\ell(\lambda)+s\ell(\lambda)$. {\em Unfortunately}, in the classical Catalan case of $(n,n+1)$-cores, the statistic $\ell+s\ell$ is {\bf not} clearly related to any of the known ``major index"-type statistics.

We expect Conjecture \ref{c:maj} to be difficult. In fact, it is just a shadow from the more general subject of $q,t$-Catalan combinatorics. To illustrate this, define the {\bf co-skew-length} of an $(a,b)$-core by $s\ell'(\lambda)=(a-1)(b-1)/2 - s\ell(\lambda)$.
\begin{conjecture}\label{c:symmetry}
Let $a<b$ be coprime. Then we have
\begin{equation}\label{eq:qtcat}
\sum_\lambda q^{\ell(\lambda)}t^{s\ell'(\lambda)} = \sum_\lambda t^{\ell(\lambda)} q^{s\ell'(\lambda)},
\end{equation}
where the sum is over $(a,b)$-cores $\lambda$.
\end{conjecture}
Either side of Equation \eqref{eq:qtcat} should be regarded as a ``rational $q,t$-Catalan number". Then Conjecture \ref{c:symmetry} is a generalization of the ``symmetry problem" for $q,t$-Catalan numbers. This problem is quite hard (see, for example, \cite{haglund}). Thus, even partial progress is welcome.

To end this section we point to some related work. The ``rational $q,t$-Catalan numbers" \eqref{eq:qtcat} have been independently defined and studied by Gorsky and Mazin \cite{GM11,GM12}, and they will also appear soon in a paper of Armstrong, Loehr, and Warrington \cite{ALW13}. The paper \cite{ALW13} will explore three different interpretations of the ``skew length" statistic. (In addition to the interpretation here, the other two are due to Gorsky--Mazin and Loehr--Warrington.) Finally, the general subject of ``rational Catalan numbers" and related structures has been studied by Armstrong, Rhoades, and Williams \cite{ARW13}.

%%%%%%%%%%%%%%%%%%%%%%%%%%%%%%%%%%%%%%%%%%%%%%%%%%%%%%%%%%%%%%%%%%%%%
%  Section
%%%%%%%%%%%%%%%%%%%%%%%%%%%%%%%%%%%%%%%%%%%%%%%%%%%%%%%%%%%%%%%%%%%%%
\section{Abacus diagrams}\label{sec:abacus}

An {\bf abacus diagram} (or simply {\bf abacus}) is a diagram containing $R$
columns labeled $1, 2, \ldots, R$, called {\bf runners}.  Runner $i$ contains
entries labeled by the integers $m R + i$ for each {\bf level} $m$ where
$-\infty < m < \infty$. 

In figures, we orient the runners vertically with entries increasing from left
to right, top to bottom.  Entries in the abacus diagram may be circled; such
circled elements are called {\bf beads}.  Entries that are not circled are
called {\bf gaps}.  We refer to the collection consisting of the lowest beads
in each runner as the {\bf defining beads} of the abacus.  We say that an
abacus is {\bf $j$-flush} if whenever position $e$ is a bead in the abacus we
have that $e-j$ is also a bead.  

As this construction essentially defines a labeling on the infinite binary string from Section~\ref{sec:co}, two abacus diagrams are equivalent if the infinite sequence of beads and gaps are the same (so that their entries only differ by a constant position).  Two common ways to standardize an abacus diagram are to make it {\em normalized} or {\em balanced}.  We say that an abacus is {\bf balanced} if the sum of the levels of the defining beads of the abacus is zero.  We say that an abacus is {\bf normalized} if the first gap occurs in position $0$.  (And when normalized, we may relabel runner $R$ as runner $0$ and place it on the left.)

Balanced, $n$-flush abacus diagrams with $n$ runners are in bijection with minimal length coset representatives of type $\widetilde{A}_{n-1}/A_{n-1}$ by interpreting the defining beads of the abacus (written in increasing order) as the entries in the base window of the corresponding affine permutation.  

Further, the classical argument of James \cite{JK} described in Section~\ref{sec:co} gives a bijection between the set of balanced $n$-flush abacus diagrams and the set of $n$-core partitions:  Given an abacus, we create a partition whose southeast boundary is the lattice path obtained by reading the entries of the abacus in increasing order and recording a north-step for each bead, and recording an east-step for each gap.

\begin{example}\label{ex:e21} Figure~\ref{fig:e21} shows the $4$-core partition $\lam=(3,3,1,1,1)$ and two corresponding abacus diagrams, which correspond to a $4$-flush abacus and to the element $[-4,1,6,7]$ in $\tld{A}_3/A_3$.  The abacus in the middle is both normalized (the first gap is in position $0$) and balanced (the levels of the defining beads on runners $1$ through $4$ are $0$, $1$, $1$, $-2$, which sum to $0$). On the right is the alternative method of drawing this same normalized abacus on runners $0$ through $3$. 
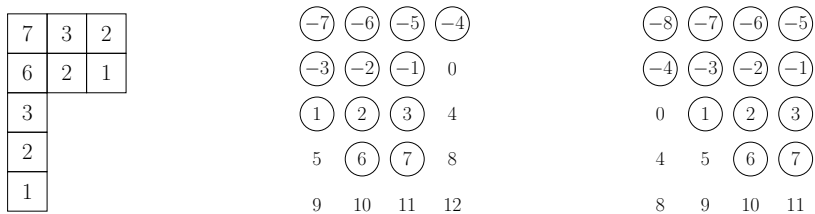
\begin{figure}[htb]
\begin{center}
\raisebox{.05in}{\scalebox{0.35}{\begin{tikzpicture}
\draw (0,-13.5cm) node [rectangle, minimum size=1.5cm, inner sep=0pt, draw, anchor=south west]  {\Huge $7$};
\draw (1.5cm,-13.5cm) node [rectangle, minimum size=1.5cm, inner sep=0pt, draw, anchor=south west]  {\Huge $3$};
\draw (3.0cm,-13.5cm) node [rectangle, minimum size=1.5cm, inner sep=0pt, draw, anchor=south west]  {\Huge $2$};
\draw (0,-15.0cm) node [rectangle, minimum size=1.5cm, inner sep=0pt, draw, anchor=south west]  {\Huge $6$};
\draw (1.5cm,-15.0cm) node [rectangle, minimum size=1.5cm, inner sep=0pt, draw, anchor=south west]  {\Huge $2$};
\draw (3.0cm,-15.0cm) node [rectangle, minimum size=1.5cm, inner sep=0pt, draw, anchor=south west]  {\Huge $1$};
\draw (0,-16.5cm) node [rectangle, minimum size=1.5cm, inner sep=0pt, draw, anchor=south west]  {\Huge $3$};
\draw (0,-18.0cm) node [rectangle, minimum size=1.5cm, inner sep=0pt, draw, anchor=south west]  {\Huge $2$};
\draw (0,-19.5cm) node [rectangle, minimum size=1.5cm, inner sep=0pt, draw, anchor=south west]  {\Huge $1$};
\end{tikzpicture}}}\qquad\qquad\qquad
\scalebox{0.3}{\begin{tikzpicture}
\draw (2cm,2cm) node [circle, minimum size=1.5cm, inner sep=0pt, draw, anchor=south west] {\Huge $-7$};
\draw (4cm,2cm) node [circle, minimum size=1.5cm, inner sep=0pt, draw, anchor=south west] {\Huge $-6$};
\draw (6cm,2cm) node [circle, minimum size=1.5cm, inner sep=0pt, draw, anchor=south west] {\Huge $-5$};
\draw (8cm,2cm) node [circle, minimum size=1.5cm, inner sep=0pt, draw, anchor=south west] {\Huge $-4$};
\draw (2cm,0cm) node [circle, minimum size=1.5cm, inner sep=0pt, draw, anchor=south west] {\Huge $-3$};
\draw (4cm,0cm) node [circle, minimum size=1.5cm, inner sep=0pt, draw, anchor=south west] {\Huge $-2$};
\draw (6cm,0cm) node [circle, minimum size=1.5cm, inner sep=0pt, draw, anchor=south west] {\Huge $-1$};
\draw (8cm,0cm) node [circle, minimum size=1.5cm, inner sep=0pt, draw=none, anchor=south west] {\Huge $0$};
\draw (2cm,-2cm) node [circle, minimum size=1.5cm, inner sep=0pt, draw, anchor=south west] {\Huge $1$};
\draw (4cm,-2cm) node [circle, minimum size=1.5cm, inner sep=0pt, draw, anchor=south west] {\Huge $2$};
\draw (6cm,-2cm) node [circle, minimum size=1.5cm, inner sep=0pt, draw, anchor=south west] {\Huge $3$};
\draw (8cm,-2cm) node [circle, minimum size=1.5cm, inner sep=0pt, draw=none, anchor=south west] {\Huge $4$};
\draw (2cm,-4cm) node [circle, minimum size=1.5cm, inner sep=0pt, draw=none, anchor=south west] {\Huge $5$};
\draw (4cm,-4cm) node [circle, minimum size=1.5cm, inner sep=0pt, draw, anchor=south west] {\Huge $6$};
\draw (6cm,-4cm) node [circle, minimum size=1.5cm, inner sep=0pt, draw, anchor=south west] {\Huge $7$};
\draw (8cm,-4cm) node [circle, minimum size=1.5cm, inner sep=0pt, draw=none, anchor=south west] {\Huge $8$};
\draw (2cm,-6cm) node [circle, minimum size=1.5cm, inner sep=0pt, draw=none, anchor=south west] {\Huge $9$};
\draw (4cm,-6cm) node [circle, minimum size=1.5cm, inner sep=0pt, draw=none, anchor=south west] {\Huge $10$};
\draw (6cm,-6cm) node [circle, minimum size=1.5cm, inner sep=0pt, draw=none, anchor=south west] {\Huge $11$};
\draw (8cm,-6cm) node [circle, minimum size=1.5cm, inner sep=0pt, draw=none, anchor=south west] {\Huge $12$};
\end{tikzpicture}}\qquad\qquad\qquad
\scalebox{0.3}{\begin{tikzpicture}
\draw (0cm,2cm) node [circle, minimum size=1.5cm, inner sep=0pt, draw, anchor=south west] {\Huge $-8$};
\draw (2cm,2cm) node [circle, minimum size=1.5cm, inner sep=0pt, draw, anchor=south west] {\Huge $-7$};
\draw (4cm,2cm) node [circle, minimum size=1.5cm, inner sep=0pt, draw, anchor=south west] {\Huge $-6$};
\draw (6cm,2cm) node [circle, minimum size=1.5cm, inner sep=0pt, draw, anchor=south west] {\Huge $-5$};
\draw (0cm,0cm) node [circle, minimum size=1.5cm, inner sep=0pt, draw, anchor=south west] {\Huge $-4$};
\draw (2cm,0cm) node [circle, minimum size=1.5cm, inner sep=0pt, draw, anchor=south west] {\Huge $-3$};
\draw (4cm,0cm) node [circle, minimum size=1.5cm, inner sep=0pt, draw, anchor=south west] {\Huge $-2$};
\draw (6cm,0cm) node [circle, minimum size=1.5cm, inner sep=0pt, draw, anchor=south west] {\Huge $-1$};
\draw (0cm,-2cm) node [circle, minimum size=1.5cm, inner sep=0pt, draw=none, anchor=south west] {\Huge $0$};
\draw (2cm,-2cm) node [circle, minimum size=1.5cm, inner sep=0pt, draw, anchor=south west] {\Huge $1$};
\draw (4cm,-2cm) node [circle, minimum size=1.5cm, inner sep=0pt, draw, anchor=south west] {\Huge $2$};
\draw (6cm,-2cm) node [circle, minimum size=1.5cm, inner sep=0pt, draw, anchor=south west] {\Huge $3$};
\draw (0cm,-4cm) node [circle, minimum size=1.5cm, inner sep=0pt, draw=none, anchor=south west] {\Huge $4$};
\draw (2cm,-4cm) node [circle, minimum size=1.5cm, inner sep=0pt, draw=none, anchor=south west] {\Huge $5$};
\draw (4cm,-4cm) node [circle, minimum size=1.5cm, inner sep=0pt, draw, anchor=south west] {\Huge $6$};
\draw (6cm,-4cm) node [circle, minimum size=1.5cm, inner sep=0pt, draw, anchor=south west] {\Huge $7$};
\draw (0cm,-6cm) node [circle, minimum size=1.5cm, inner sep=0pt, draw=none, anchor=south west] {\Huge $8$};
\draw (2cm,-6cm) node [circle, minimum size=1.5cm, inner sep=0pt, draw=none, anchor=south west] {\Huge $9$};
\draw (4cm,-6cm) node [circle, minimum size=1.5cm, inner sep=0pt, draw=none, anchor=south west] {\Huge $10$};
\draw (6cm,-6cm) node [circle, minimum size=1.5cm, inner sep=0pt, draw=none, anchor=south west] {\Huge $11$};
\end{tikzpicture}}
\end{center}
\caption{The $4$-core partition $\lam=(3,3,1,1,1)$ and two corresponding abacus diagrams from Example~\ref{ex:e21}.}
\label{fig:e21}
\end{figure}
\end{example}

\medskip
In \cite{HJ12}, the second and third authors introduced an abacus diagram model
with $R=2n$ runners to represent minimal length coset representatives of type
$\widetilde{C}_n/C_n$.  In this type-$C$ abacus model,  we use $N=2n+1$
implicit labels per row so that the linear ordering of the entries of the
abacus are given by the labels $m N + i$ for level $m \in \Z$ and runner $1
\leq i \leq 2n$.  (Under these conventions, there are no entries in any type
$C$ abacus having labels $\{m N : m \in \mathbb{Z}\}$.) 

We also impose a stricter definition of {\bf balanced} on a type-$C$
abacus---the level of the defining bead on runner $i$ is the negative of the
level of the defining bead on runner $N-i$. This imposes an antisymmetry on
type-$C$ abaci where entry $N-b$ is a bead if and only if entry $N+b$ is a gap.
Restricting James's bijection gives a bijection between the set of type-$C$
balanced $2n$-flush abacus diagrams and the set of self-conjugate $2n$-core
partitions.  Under this construction, we can then interpret the defining beads
of abacus written in increasing order as the corresponding minimal length coset
representative of $\tld{C}_n/C_n$ written in one-line notation as a mirrored
$\mathbb{Z}$-permutation, just as in type $A$.

Throughout this paper, we work in types $A$ and $C$ simultaneously by letting $N$ be the number of implicit labels used on each row of the abacus, so $N = n$ in type $A$ and $N = 2n+1$ in type $C$.  We also let $R$ be the number of runners in the abacus, so $R = n$ in type $A$ and $R = 2n$ in type $C$.  

\begin{example}\label{ex:e22}
In $\widetilde{C}_2$, consider the mirrored $\mathbb{Z}$-permutation determined by $[w(1), w(2), w(3), w(4)] = [-2, 1, 4, 7]$.  The corresponding abacus diagram is given in Figure~\ref{fig:e22} and the corresponding self-conjugate $4$-core partition is $(2,1)$.
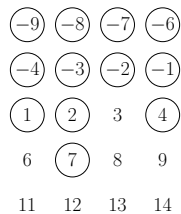
\begin{figure}[htb]
\[
\scalebox{0.3}{\begin{tikzpicture}
\draw (0cm,2cm) node [circle, minimum size=1.5cm, inner sep=0pt, draw, anchor=south west] {\Huge $-9$};
\draw (2cm,2cm) node [circle, minimum size=1.5cm, inner sep=0pt, draw, anchor=south west] {\Huge $-8$};
\draw (4cm,2cm) node [circle, minimum size=1.5cm, inner sep=0pt, draw, anchor=south west] {\Huge $-7$};
\draw (6cm,2cm) node [circle, minimum size=1.5cm, inner sep=0pt, draw, anchor=south west] {\Huge $-6$};
\draw (0cm,0cm) node [circle, minimum size=1.5cm, inner sep=0pt, draw, anchor=south west] {\Huge $-4$};
\draw (2cm,0cm) node [circle, minimum size=1.5cm, inner sep=0pt, draw, anchor=south west] {\Huge $-3$};
\draw (4cm,0cm) node [circle, minimum size=1.5cm, inner sep=0pt, draw, anchor=south west] {\Huge $-2$};
\draw (6cm,0cm) node [circle, minimum size=1.5cm, inner sep=0pt, draw, anchor=south west] {\Huge $-1$};
\draw (0cm,-2cm) node [circle, minimum size=1.5cm, inner sep=0pt, draw, anchor=south west] {\Huge $1$};
\draw (2cm,-2cm) node [circle, minimum size=1.5cm, inner sep=0pt, draw, anchor=south west] {\Huge $2$};
\draw (4cm,-2cm) node [circle, minimum size=1.5cm, inner sep=0pt, draw=none, anchor=south west] {\Huge $3$};
\draw (6cm,-2cm) node [circle, minimum size=1.5cm, inner sep=0pt, draw, anchor=south west] {\Huge $4$};
\draw (0cm,-4cm) node [circle, minimum size=1.5cm, inner sep=0pt, draw=none, anchor=south west] {\Huge $6$};
\draw (2cm,-4cm) node [circle, minimum size=1.5cm, inner sep=0pt, draw, anchor=south west] {\Huge $7$};
\draw (4cm,-4cm) node [circle, minimum size=1.5cm, inner sep=0pt, draw=none, anchor=south west] {\Huge $8$};
\draw (6cm,-4cm) node [circle, minimum size=1.5cm, inner sep=0pt, draw=none, anchor=south west] {\Huge $9$};
\draw (0cm,-6cm) node [circle, minimum size=1.5cm, inner sep=0pt, draw=none, anchor=south west] {\Huge $11$};
\draw (2cm,-6cm) node [circle, minimum size=1.5cm, inner sep=0pt, draw=none, anchor=south west] {\Huge $12$};
\draw (4cm,-6cm) node [circle, minimum size=1.5cm, inner sep=0pt, draw=none, anchor=south west] {\Huge $13$};
\draw (6cm,-6cm) node [circle, minimum size=1.5cm, inner sep=0pt, draw=none, anchor=south west] {\Huge $14$};
\end{tikzpicture}}
\]
\caption{The abacus diagram from Example~\ref{ex:e22}.}
\label{fig:e22}
\end{figure}
\end{example}

%%%%%%%%%%%%%%%%%%%%%%%%%%%%%%%%%%%%%%%%%%%%%%%%%%%%%%%%%%%%%%%%%%%%%
%  Section
%%%%%%%%%%%%%%%%%%%%%%%%%%%%%%%%%%%%%%%%%%%%%%%%%%%%%%%%%%%%%%%%%%%%%
\section{Regions of the $m$-Shi arrangement} \label{sec:geometry}

Consider the root system of type $A_{n-1}$ or type $C_n$ embedded in a Euclidean space
$V = \R^n$ with inner product $(\cdot,\cdot)$ and orthonormal basis $\{\e_1,
\ldots, \e_n\}$.  Then the {\em $m$-Shi hyperplane arrangement} consists of $v
\in V$ such that 
\[ -m < (v,\alpha) \leq m \text{ for all positive roots $\alpha$. } \]

For example, the hyperplanes in type~$A_{n-1}$ consist of $v = \sum_{i=1}^n v_i \e_i$ such that
\[ v_i - v_j \in \{0, 1, \ldots, m\} \text{ for $1 \leq i < j \leq n$, } \]
while for type~$C_n$ we additionally have 
\[ v_i + v_j \in \{0, 1, \ldots, m\} \text{ for $1 \leq i < j \leq n$, and } \]
\[ 2v_i \in \{0, 1, \ldots, m\} \text{ for $1 \leq i \leq n$ }. \]
In this work, we restrict to the dominant cone $\{ v \in V : (v,\alpha) \geq 0 $ for all positive roots $\alpha\}$.
Pictures from rank $2$ are shown in Figures~\ref{f:A2} and \ref{f:C2}.  For the $A_2$ figures, we draw the restriction of the arrangement to the plane $v_1 + v_2 + v_3 = 0$ in $\R^3$.  

\begin{figure}[ht]
\begin{center}
\begin{tikzpicture}
    \pgfmathsetmacro\am{6}  % max
    \pgfmathsetmacro\aw{2}  % width
    \pgfmathsetmacro\tsx{1}  % text shift
    \pgfmathsetmacro\tsy{0.8}  % text shift

    \draw[blue] (0,0) -- (\am,0);
    \draw[blue] (0,0) -- (60:\am);
    \draw[blue] (60:\aw) -- +(\am,0);
    \draw[blue] (\aw,0) -- +(60:\am);
    \draw[blue] (60:\aw) -- (\aw,0);
    \draw[blue] (60:2*\aw) -- (2*\aw,0);
    \draw[blue] (60:2*\aw) -- +(\am,0);
    \draw[blue] (2*\aw,0) -- +(60:\am);

    \draw (0,0)+(0.5*\aw,0) node [rectangle,fill=white]{\color{blue} $s_2$};
    \draw (60:0.5*\aw) node [rectangle,fill=white]{\color{blue} $s_1$};
    \draw (30:0.88*\aw) node [rectangle,fill=white]{\color{blue} $s_0$};

    \draw (0,-0.2)+(\tsx,\tsy) node {$1\;2\;3$};
    \draw (0.5*\aw,0.4)+(\tsx,\tsy) node {$0\;2\;4$};
    \draw (\aw,-0.2)+(\tsx,\tsy) node {$\bar{1}\;3\;4$};
    \draw (1.5*\aw,0.4)+(\tsx,\tsy) node {$\bar{2}\;3\;5$};
    \draw (2*\aw,-0.2)+(\tsx,\tsy) node {$\bar{3}\;4\;5$};

    \draw (0.5*\aw,\aw-0.4)+(\tsx,\tsy) node {$0\;1\;5$};
    \draw (\aw,\aw+0.1)+(\tsx,\tsy) node {$\bar{1}\;1\;6$};
    \draw (1.5*\aw,\aw-0.4)+(\tsx,\tsy) node {$\bar{2}\;2\;6$};
    \draw (2*\aw,\aw+0.1)+(\tsx,\tsy) node {$\bar{3}\;2\;7$};
    \draw (2.5*\aw,\aw-0.4)+(\tsx,\tsy) node {$\bar{4}\;3\;7$};

    \draw (\aw,2*\aw-0.8)+(\tsx,\tsy) node {$\bar{1}\;0\;7$};
    \draw (2*\aw,2*\aw-0.8)+(\tsx,\tsy) node {$\bar{3}\;1\;8$};
    \draw (3*\aw,2*\aw-0.8)+(\tsx,\tsy) node {$\bar{5}\;2\;9$};
\end{tikzpicture}
\end{center}
\caption{The $m=2$ Shi arrangement in $A_2$}\label{f:A2}
\end{figure}
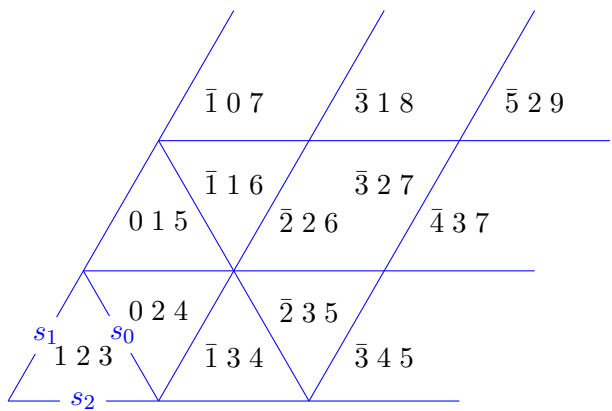

Recall from \cite[Chapter 4]{humphreys} that the affine Weyl group acts on $V$
by reflections, and so $V$ decomposes into {\bf alcoves} that are the connected
components of the complement of the set of hyperplanes orthogonal to positive
roots together with all their parallel translates.  Denote the fundamental
alcove by $\A$.  Then there is a simply transitive action of the affine Weyl
group on the set of alcoves.  In fact, there are two actions:  the left action
reflects an alcove (non-locally) across one of the hyperplanes through the
origin, while the right action reflects an alcove (locally) across one of its
bounding hyperplanes.  Then, we have that $w \A$ corresponds to a dominant
alcove if and only if $w$ is ``left grassmannian'' in the sense that $D_L(w)
\subseteq \{s_0\}$, where $D_L$ denotes the left descent set.
We label each alcove with the defining beads of the corresponding balanced
abacus, where a bar over a number represents the negative of that number.

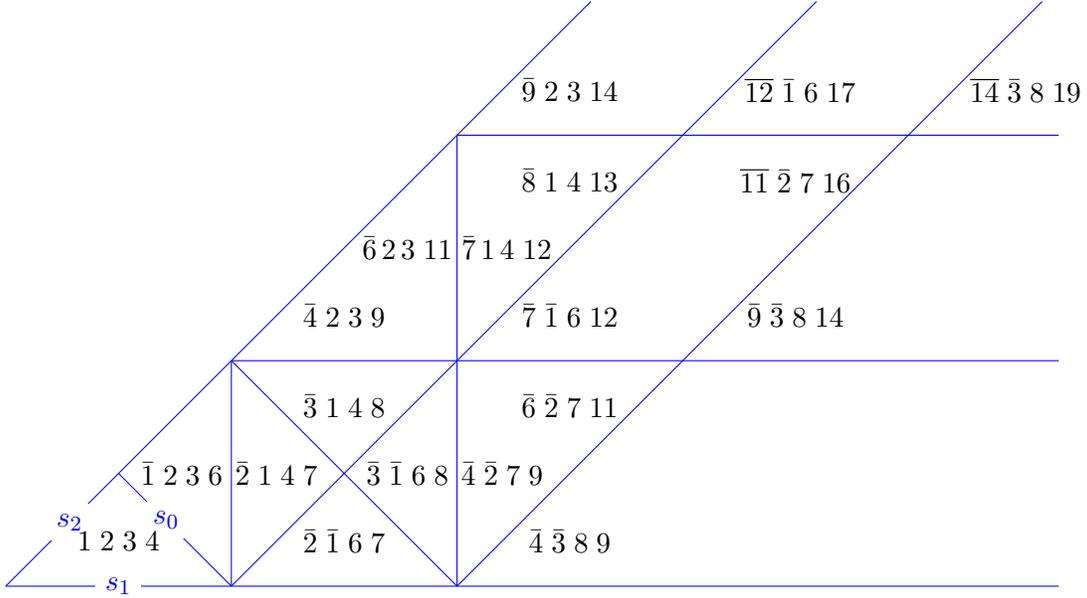
\begin{figure}[t]
\begin{center}
\begin{tikzpicture}
    \pgfmathsetmacro\am{11}  % max
    \pgfmathsetmacro\arw{4.24264}  % root 2 width
    \pgfmathsetmacro\aw{3}  % width
    \pgfmathsetmacro\tsx{1.5}  % text shift
    \pgfmathsetmacro\tsy{0.6}  % text shift

    \draw[blue] (0,0) -- (\am+\aw,0);
    \draw[blue] (0,0) -- (45:\am);
    \draw[blue] (45:\arw) -- +(\am,0);
    \draw[blue] (45:2*\arw) -- +(\am-\aw,0);
    \draw[blue] (\aw,0) -- +(45:\am);
    \draw[blue] (2*\aw,0) -- +(45:\am);
    \draw[blue] (\aw,0) -- +(90:\aw);
    \draw[blue] (2*\aw,0) -- +(90:2*\aw);
    \draw[blue] (\aw,0) -- +(135:0.5*\arw);
    \draw[blue] (2*\aw,0) -- +(135:\arw);

    \draw (0,0)+(0.5*\aw,0) node [rectangle,fill=white] {\color{blue} $s_1$};
    \draw (45:0.4*\aw) node [rectangle,fill=white] {\color{blue} $s_2$};
    \draw (22.5:0.77*\aw) node [rectangle,fill=white] {\color{blue} $s_0$};

    \draw (0,0)+(\tsx,\tsy) node {$1\;2\;3\;4$};
    \draw (0.28*\aw,0.9)+(\tsx,\tsy) node {$\bar{1}\;2\;3\;6$};
    \draw (0.7*\aw,0.9)+(\tsx,\tsy) node {$\bar{2}\;1\;4\;7$};
    \draw (\aw,0)+(\tsx,\tsy) node {$\bar{2}\;\bar{1}\;6\;7$};
    \draw (\aw,1.8)+(\tsx,\tsy) node {$\bar{3}\;1\;4\;8$};
    \draw (1.28*\aw,0.9)+(\tsx,\tsy) node {$\bar{3}\;\bar{1}\;6\;8$};
    \draw (1.7*\aw,0.9)+(\tsx,\tsy) node {$\bar{4}\;\bar{2}\;7\;9$};
    \draw (2*\aw,0)+(\tsx,\tsy) node {$\bar{4}\;\bar{3}\;8\;9$};
    \draw (2*\aw,1.8)+(\tsx,\tsy) node {$\bar{6}\;\bar{2}\;7\;11$};

    \draw (\aw,\aw+0)+(\tsx,\tsy) node {$\bar{4}\;2\;3\;9$};
    \draw (1.28*\aw,\aw+0.9)+(\tsx,\tsy) node {$\bar{6}\,2\,3\;11$};
    \draw (1.72*\aw,\aw+0.9)+(\tsx,\tsy) node {$\bar{7}\,1\,4\;12$};
    \draw (2*\aw,\aw+0)+(\tsx,\tsy) node {$\bar{7}\;\bar{1}\;6\;12$};
    \draw (2*\aw,\aw+1.8)+(\tsx,\tsy) node {$\bar{8}\;1\;4\;13$};

    \draw (3*\aw,\aw+1.8)+(\tsx,\tsy) node {$\overline{11}\;\bar{2}\;7\;16$};
    \draw (3*\aw,\aw)+(\tsx,\tsy) node {$\bar{9}\;\bar{3}\;8\;14$};

    \draw (2*\aw,2*\aw)+(\tsx,\tsy) node {$\bar{9}\;2\;3\;14$};
    \draw (3.02*\aw,2*\aw)+(\tsx,\tsy) node {$\overline{12}\;\bar{1}\;6\;17$};
    \draw (4.02*\aw,2*\aw)+(\tsx,\tsy) node {$\overline{14}\;\bar{3}\;8\;19$};
\end{tikzpicture}
\end{center}
\caption{The $m=2$ Shi arrangement in $C_2$}\label{f:C2}
\end{figure}

We can read off the right (local) descents and the left (non-local) inner
products of $w \A$ simultaneously from the abacus, a result of Shi \cite[Theorem~4.1]{Shi99}.

\begin{lemma} \label{l:inner_product}
Let $w(i)$ denote the $i$-th entry of the one-line notation for $w$, or
equivalently, the position of the $i$-th defining bead in the balanced abacus for $w$ in
type~$A$ or $C$.  Then the inner product of any point in $w \A$ with a positive
root $\alpha$ satisfies
\[ \bigg\lf \frac{w(j)-w(i)}{N} \bigg\rf < (w\A,\alpha) < \bigg\lf \frac{w(j)-w(i)}{N} \bigg\rf + 1, \]
for some $1 \leq i<j \leq R$.
\end{lemma}
\begin{proof}
This follows from the symmetry of the abacus and \cite[Theorem 4.1]{HJ12}.
It is straightforward to work out the correspondence between differences and
positive roots.
For example, the inner product with $e_i+e_j$ in type~$C$ can be realized as
$\lf \frac{w(n+j)-w(i)}{N} \rf$.
\end{proof}

\begin{lemma}\label{l:right_descent}
In types $A$ and $C$, $w$ has $s_i$ as a right descent if and only if
the position of the defining bead in column $i+1$ is at least $N$ plus the
position of the defining bead in column $i$ (where column $0$ is interpreted as
column $2n$ in type~$C$, and column $n$ in type~$A$).
\end{lemma}
\begin{proof}
This follows from \cite[Section 3.2]{HJ12}.
\end{proof}

As a consequence of these lemmas, it is possible to read off the ``Shi
coordinates'' of an alcove that specify the number of translations in each
positive root direction.

\begin{example}
    The alcove labeled $\bar{1}16$ in Figure~\ref{f:A2} lies 
    \begin{itemize}
        \item $\big\lf \frac{1-(-1)}{3} \big\rf = 0$ translates past the hyperplane labeled $s_1$
        \item $\big\lf \frac{6-1}{3} \big\rf = 1$ translate past the hyperplane labeled $s_2$
        \item $\big\lf \frac{6-(-1)}{3} \big\rf = 2$ translates past the hyperplane labeled $s_0$.
    \end{itemize}
    If we reflected the labeling of the boundary of $\A$ by $w$, we would find
    that $s_2$ is the unique right descent of $\bar{1}16$, corresponding to the
    fact that $6$ is the position of the defining bead in column $3$ and this
    lies at least $N=3$ positions past the defining bead in column $2$.
\end{example}

\begin{example}
    The alcove labeled $\bar{4}\bar{2}79$ in Figure~\ref{f:C2} lies 
    \begin{itemize}
        \item $\big\lf \frac{(-2)-(-4)}{5} \big\rf = 0$ translates past the hyperplane labeled $s_1$
        \item $\big\lf \frac{7-(-2)}{5} \big\rf = 1$ translate past the hyperplane labeled $s_2$
        \item $\big\lf \frac{9-(-4)}{5} \big\rf = 2$ translates past the hyperplane labeled $s_0$
        \item $\big\lf \frac{7-(-4)}{5} \big\rf = 2$ translates past the hyperplane perpendicular to the remaining positive root $e_1+e_2$.
    \end{itemize}
    The unique right descent for this element is $s_1$.
\end{example}

We say that a dominant alcove is {\bf $m$-minimal} if it is the unique alcove of minimal
length in its region of the $m$-Shi arrangement.  We say that a dominant alcove
is {\bf $m$-bounded} if it is the unique alcove of maximal length in its region of the
$m$-Shi arrangement.  The uniqueness of these alcoves was shown in \cite{Ath}.
Our main result in this section is the following theorem.  A less explicit proof for type
$A$ is given in \cite{FV1}.

\begin{theorem}\label{thm:alcove}
In types $A$ and $C$, a dominant alcove is $m$-minimal if and only if the
corresponding abacus diagram is $(Rm+1)$-flush.
Moreover, a dominant alcove is $m$-bounded if and only if the corresponding
abacus diagram is $(Rm-1)$-flush.
\end{theorem}
\begin{proof}
Let $w\A$ be a dominant alcove.  Then $w\A$ is $m$-minimal if it is the 
minimal length alcove in its region of the $m$-Shi arrangement.  Equivalently,
for each descent $s_i$ of $w$, we must have that $w$ and $ws_i$ are separated by an
$m$-Shi hyperplane.  Contrapositively, there do not exist two defining beads in
the abacus for $w$ that form a right descent and contribute a left inner product
with a positive root that is greater than $m$.

By Lemmas~\ref{l:inner_product} and \ref{l:right_descent}, this means that the
one-line notation for $w$ never contains $1 \leq i < j \leq R$ with
\[ \bigg\lf \frac{w(j)-w(i)}{N} \bigg\rf > m \]
and 
\[ \big(w(j) \mod N\big) \equiv \big(w(i) \mod N\big) + 1 \ \pmod R. \]
But this is precisely equivalent to requiring that the abacus be $(Rm+1)$-flush.

Similarly, $w\A$ is $m$-bounded if it is the maximal length alcove in
its region of the $m$-Shi arrangement.  This is equivalent to requiring that
there do not exist two defining beads in $w$ that form a right ascent and have a
left inner product with the corresponding positive root that is greater than
$m$.  Once again by Lemmas~\ref{l:inner_product} and \ref{l:right_descent}, this
means that the one-line notation for $w$ never contains $1 \leq i < j \leq R$ with
\[ \bigg\lf \frac{w(j)-w(i)}{N} \bigg\rf > m \]
and 
\[ \big(w(j) \mod N\big) \equiv \big(w(i) \mod N\big) - 1 \ \pmod R. \]
This is precisely equivalent to requiring that the abacus be $(Rm-1)$-flush.
\end{proof}

\begin{example}
    In type~$A_2$, $m = 1$, the $m$-minimal alcoves correspond to
    \[ 123, 024, 015, \bar{1}34, \bar{2}26. \]

    The element $\bar{2}35$ is not minimal because the $\bar{2}$ and $5$ are off
    by $2 > m$ levels, and they form a descent since $\bar{2}$ lies on column
    $(-2 \mod 3) = 1$ of the abacus while $5$ lies on column $(5 \mod 3) = 2$.

    The element $\bar{2}26$ is minimal because although $\bar{2}$ and $6$ are
    off by $2 > m$ levels, we find that these two entries do not form a descent
    on the abacus.
\end{example}

\begin{example}
    In type~$C_2$, $m = 1$, the $m$-minimal alcoves correspond to
    \[ 1234, \bar{1}236, \bar{2}147, \bar{2}\bar{1}67, \bar{4}239,
    \bar{7}\bar{1}6\hspace{.02in}12. \]
\end{example}

\begin{corollary}
The $m$-minimal alcoves in the $m$-Shi arrangement of type~$C_n$ are in bijection
with self-conjugate partitions that are $(2n)$-core and $(2nm+1)$-core.
The $m$-bounded alcoves in the $m$-Shi arrangement of type~$C_n$ are in bijection
with self-conjugate partitions that are $(2n)$-core and $(2nm-1)$-core.
\end{corollary}

Applying Ford, Mai, and Sze's formula \cite{FMS}, we recover that there are
$\binom{nm+n}{n}$ dominant regions and $\binom{nm+n-1}{n}$ bounded regions in
the $m$-Shi arrangement of type~$C_n$.  This agrees with Athanasiadis's result
\cite[Corollary 1.3]{AthBull}.

%%%%%%%%%%%%%%%%%%%%%%%%%%%%%%%%%%%%%%%%%%%%%%%%%%%%%%%%%%%%%%%%%%%%%
%  Section
%%%%%%%%%%%%%%%%%%%%%%%%%%%%%%%%%%%%%%%%%%%%%%%%%%%%%%%%%%%%%%%%%%%%%
\section{A major index statistic on simultaneous core partitions}
\label{sec:maj}

In this section, we define a major index statistic that gives the $q$-analog of the
Coxeter-Catalan numbers, $\Cat_q(A_{n-1})$ and $\Cat_q(C_n)$, for simultaneous
core partitions in types $A$ and $C$, respectively.

\begin{definition}
Let $\lam$ be a simultaneous $(n,n+1)$-core partition. 
 Create the sequence
 $x=(x_0,\hdots,x_{n-1})$ where $x_i$ equals the number of boxes in the first column of $\lam$ whose hook length modulo $n$ equals $i$. (Note  $x_0=0$ always.) Define 
\begin{equation}\label{eq:majA}
\maj_A(\lam)=\sum_{i\,:\,x_{i-1}\geq x_{i}} (2i-x_i).
\end{equation}

Let $\lam$ be a self-conjugate simultaneous $(2n,2n+1)$-core partition. 
We define the set $\calW$ of {\bf diagonal arm lengths} $\{w_1,w_2,\hdots,w_k\}$ where $w_i$ is one more than the number of boxes to the right of the $i$-th box on the diagonal of $\lam$.
Create the sequence $x=(x_0,x_1,\hdots,x_n)$ where $x_0=0$ and
\[x_i=\lvert\{w\in\calW:w\bmod 2n\equiv i\}\rvert-\lvert\{w\in\calW:w\bmod 2n\equiv 2n-i+1\}\rvert\]
for $1\leq i\leq n$. 
Define 
\begin{equation}\label{eq:majC}
\maj_C(\lam)=2\sum_{i\,:\,x_{i-1}\geq x_{i}} (2i-x_i-1).
\end{equation}
\end{definition}

\begin{remark}
Similar to the definition of the major index statistic of a permutation, these sums are over the positions of the weak descents in a sequence.  Since our sequence $x$ starts with $x_0$, our definition of position of a weak descent is the value $i$ such that $x_{i-1}\geq x_i$.  This allows for possible weak descents to occur in positions $1$ through $n-2$ in type~$A$ and in positions $1$ through $n-1$ in type~$C$.

In terms of the abacus diagram, $x_i$ is the level of the defining bead on runner $i$ (in type~$A$ the abacus must be normalized first), so the definitions of $\maj_A$ and $\maj_C$ can also be applied directly to the corresponding abacus diagram.
\end{remark}

\begin{example}
For the $(7,8)$-core partition $\lam=(7,6,2,2,2,2)$, the hook lengths of the boxes in the first column of $\lam$ are $\{12,10,5,4,3,2\}$, which modulo $7$ equals $\{5,3,5,4,3,2\}$.  We conclude that $x=(0,0,1,2,1,2,0)$ with weak descents in positions $1$, $4$, and $6$.  As such, \[\maj_A(\lam)=(2\cdot 1-0)+(2\cdot 4-1)+(2\cdot 6-0)=21.\]
\end{example}

\begin{example}
For the self-conjugate $(14,15)$-core partition
\[\mu=(19,19,16,12,9,9,9,7,7,4,4,4,3,3,3,3,2,2,2),\] the set of diagonal arm lengths is $\{19,18,14,9,5,4,3\}$, which modulo $14$ gives $\{5,4,14,9,5,4,3\}$.  Therefore we have $x=(0,-1,0,1,2,2,-1,0)$ with weak descents in positions $1$, $5$, and $6$.  We find that 
\[\maj_C(\mu)=2\big[\big(2\cdot 1-1-(-1)\big)+\big(2\cdot 5-1-2\big)+\big(2\cdot 6-1-(-1)\big)\big]=42.\]
\end{example}

The main result in this section is that

\begin{theorem}\label{t:qA} The major index statistic defined above gives a $q$-analog of the type~$A$ and type~$C$ Catalan numbers.  In particular,
\[\sum_{\substack{\textup{$\lam$ is an}\\\textup{$(n,n+1)$-core}}} q^{\maj_A(\lam)}=\frac{1}{[n+1]_q}\qbinom{2n}{n}
\qquad\textup{ and }\qquad
\sum_{\substack{\textup{$\lam$ is a self-conj.}\\\textup{$(2n,2n+1)$-core}}}
q^{\maj_C(\lam)}=\qqbinom{2n}{n}.\]
\end{theorem}

The proof of this result is given below; it uses the bijections between simultaneous core partitions, abacus diagrams, and lattice paths.  In type~$A$, we quickly revisit and then apply Anderson's bijection \cite[Proposition~1]{Anderson}.  In type~$C$, Ford, Mai, and Sze \cite{FMS} developed a lattice path method to count self-conjugate $(a,b)$-core partitions for $a<b$ relatively prime.  Hanusa and Jones's abacus model for type~$C$ \cite{HJ12} streamlines this bijection and helps to develop further intuition about it.% about its analogy with Anderson's bijection.

\begin{proposition}[Proposition 1, \cite{Anderson}]\label{p:Anderson}
There exists a bijection:
\[
\calL:\left\{\begin{tabular}{c}$a$-flush and $b$-flush\\abacus diagrams\end{tabular}\right\}
\longleftrightarrow 
\left\{\begin{tabular}{c}$N$-$E$ lattice paths\\ $(0,0)\row(b,a)$\\on or above $y=\frac{a}{b}x$\end{tabular}\right\}.
\]
\end{proposition}

\begin{proof}
We will apply a vertical reflection to Anderson's original bijection; see Figure~\ref{fig:Anderson}.

Organize the integers into a square lattice; in the box with corners $(i,j)$ and $(i+1,j+1)$, place the integer $-a(i+1)+bj$.  For a {\em normalized} abacus $A$ which is both $a$-flush and $b$-flush, the dividing lattice path between integers that are beads of $A$ and integers that are gaps of $A$ is a lattice path $\calL(A)$ from $(0,0)$ to $(b,a)$ on or above the diagonal.  The reason this works is that under this assignment of integers, the $a$-flush condition can be read horizontally and the $b$-flush condition can be read vertically.  

Conversely, for a $N$-$E$ lattice path $L:(0,0)\row(b,a)$, interpret the number to the right of an up step to be a defining bead on the $a$-flush abacus, and the number below a right step to be a defining bead on the $b$-flush abacus.  
\end{proof}

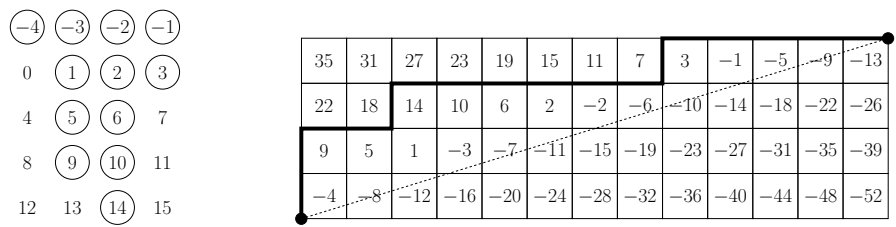
\begin{figure}%[h]
\begin{center}
\scalebox{0.3}{\begin{tikzpicture}
\draw (0cm,0cm) node [circle, minimum size=1.5cm, inner sep=0pt, draw, anchor=south west] {\Huge $-4$};
\draw (2cm,0cm) node [circle, minimum size=1.5cm, inner sep=0pt, draw, anchor=south west] {\Huge $-3$};
\draw (4cm,0cm) node [circle, minimum size=1.5cm, inner sep=0pt, draw, anchor=south west] {\Huge $-2$};
\draw (6cm,0cm) node [circle, minimum size=1.5cm, inner sep=0pt, draw, anchor=south west] {\Huge $-1$};
\draw (0cm,-2cm) node [circle, minimum size=1.5cm, inner sep=0pt, draw=none, anchor=south west] {\Huge $0$};
\draw (2cm,-2cm) node [circle, minimum size=1.5cm, inner sep=0pt, draw, anchor=south west] {\Huge $1$};
\draw (4cm,-2cm) node [circle, minimum size=1.5cm, inner sep=0pt, draw, anchor=south west] {\Huge $2$};
\draw (6cm,-2cm) node [circle, minimum size=1.5cm, inner sep=0pt, draw, anchor=south west] {\Huge $3$};
\draw (0cm,-4cm) node [circle, minimum size=1.5cm, inner sep=0pt, draw=none, anchor=south west] {\Huge $4$};
\draw (2cm,-4cm) node [circle, minimum size=1.5cm, inner sep=0pt, draw, anchor=south west] {\Huge $5$};
\draw (4cm,-4cm) node [circle, minimum size=1.5cm, inner sep=0pt, draw, anchor=south west] {\Huge $6$};
\draw (6cm,-4cm) node [circle, minimum size=1.5cm, inner sep=0pt, draw=none, anchor=south west] {\Huge $7$};
\draw (0cm,-6cm) node [circle, minimum size=1.5cm, inner sep=0pt, draw=none, anchor=south west] {\Huge $8$};
\draw (2cm,-6cm) node [circle, minimum size=1.5cm, inner sep=0pt, draw, anchor=south west] {\Huge $9$};
\draw (4cm,-6cm) node [circle, minimum size=1.5cm, inner sep=0pt, draw, anchor=south west] {\Huge $10$};
\draw (6cm,-6cm) node [circle, minimum size=1.5cm, inner sep=0pt, draw=none, anchor=south west] {\Huge $11$};
\draw (0cm,-8cm) node [circle, minimum size=1.5cm, inner sep=0pt, draw=none, anchor=south west] {\Huge $12$};
\draw (2cm,-8cm) node [circle, minimum size=1.5cm, inner sep=0pt, draw=none, anchor=south west] {\Huge $13$};
\draw (4cm,-8cm) node [circle, minimum size=1.5cm, inner sep=0pt, draw, anchor=south west] {\Huge $14$};
\draw (6cm,-8cm) node [circle, minimum size=1.5cm, inner sep=0pt, draw=none, anchor=south west] {\Huge $15$};
\end{tikzpicture}}\qquad\qquad
\scalebox{0.30}{\begin{tikzpicture}
\draw (0cm,0cm) node [rectangle, minimum size=2cm, inner sep=0pt, draw, anchor=south west] {\Huge$-4$};
\draw (2cm,0cm) node [rectangle, minimum size=2cm, inner sep=0pt, draw, anchor=south west] {\Huge$-8$};
\draw (4cm,0cm) node [rectangle, minimum size=2cm, inner sep=0pt, draw, anchor=south west] {\Huge$-12$};
\draw (6cm,0cm) node [rectangle, minimum size=2cm, inner sep=0pt, draw, anchor=south west] {\Huge$-16$};
\draw (8cm,0cm) node [rectangle, minimum size=2cm, inner sep=0pt, draw, anchor=south west] {\Huge$-20$};
\draw (10cm,0cm) node [rectangle, minimum size=2cm, inner sep=0pt, draw, anchor=south west] {\Huge$-24$};
\draw (12cm,0cm) node [rectangle, minimum size=2cm, inner sep=0pt, draw, anchor=south west] {\Huge$-28$};
\draw (14cm,0cm) node [rectangle, minimum size=2cm, inner sep=0pt, draw, anchor=south west] {\Huge$-32$};
\draw (16cm,0cm) node [rectangle, minimum size=2cm, inner sep=0pt, draw, anchor=south west] {\Huge$-36$};
\draw (18cm,0cm) node [rectangle, minimum size=2cm, inner sep=0pt, draw, anchor=south west] {\Huge$-40$};
\draw (20cm,0cm) node [rectangle, minimum size=2cm, inner sep=0pt, draw, anchor=south west] {\Huge$-44$};
\draw (22cm,0cm) node [rectangle, minimum size=2cm, inner sep=0pt, draw, anchor=south west] {\Huge$-48$};
\draw (24cm,0cm) node [rectangle, minimum size=2cm, inner sep=0pt, draw, anchor=south west] {\Huge$-52$};
\draw (0cm,2cm) node [rectangle, minimum size=2cm, inner sep=0pt, draw, anchor=south west] {\Huge$9$};
\draw (2cm,2cm) node [rectangle, minimum size=2cm, inner sep=0pt, draw, anchor=south west] {\Huge$5$};
\draw (4cm,2cm) node [rectangle, minimum size=2cm, inner sep=0pt, draw, anchor=south west] {\Huge$1$};
\draw (6cm,2cm) node [rectangle, minimum size=2cm, inner sep=0pt, draw, anchor=south west] {\Huge$-3$};
\draw (8cm,2cm) node [rectangle, minimum size=2cm, inner sep=0pt, draw, anchor=south west] {\Huge$-7$};
\draw (10cm,2cm) node [rectangle, minimum size=2cm, inner sep=0pt, draw, anchor=south west] {\Huge$-11$};
\draw (12cm,2cm) node [rectangle, minimum size=2cm, inner sep=0pt, draw, anchor=south west] {\Huge$-15$};
\draw (14cm,2cm) node [rectangle, minimum size=2cm, inner sep=0pt, draw, anchor=south west] {\Huge$-19$};
\draw (16cm,2cm) node [rectangle, minimum size=2cm, inner sep=0pt, draw, anchor=south west] {\Huge$-23$};
\draw (18cm,2cm) node [rectangle, minimum size=2cm, inner sep=0pt, draw, anchor=south west] {\Huge$-27$};
\draw (20cm,2cm) node [rectangle, minimum size=2cm, inner sep=0pt, draw, anchor=south west] {\Huge$-31$};
\draw (22cm,2cm) node [rectangle, minimum size=2cm, inner sep=0pt, draw, anchor=south west] {\Huge$-35$};
\draw (24cm,2cm) node [rectangle, minimum size=2cm, inner sep=0pt, draw, anchor=south west] {\Huge$-39$};
\draw (0cm,4cm) node [rectangle, minimum size=2cm, inner sep=0pt, draw, anchor=south west] {\Huge$22$};
\draw (2cm,4cm) node [rectangle, minimum size=2cm, inner sep=0pt, draw, anchor=south west] {\Huge$18$};
\draw (4cm,4cm) node [rectangle, minimum size=2cm, inner sep=0pt, draw, anchor=south west] {\Huge$14$};
\draw (6cm,4cm) node [rectangle, minimum size=2cm, inner sep=0pt, draw, anchor=south west] {\Huge$10$};
\draw (8cm,4cm) node [rectangle, minimum size=2cm, inner sep=0pt, draw, anchor=south west] {\Huge$6$};
\draw (10cm,4cm) node [rectangle, minimum size=2cm, inner sep=0pt, draw, anchor=south west] {\Huge$2$};
\draw (12cm,4cm) node [rectangle, minimum size=2cm, inner sep=0pt, draw, anchor=south west] {\Huge$-2$};
\draw (14cm,4cm) node [rectangle, minimum size=2cm, inner sep=0pt, draw, anchor=south west] {\Huge$-6$};
\draw (16cm,4cm) node [rectangle, minimum size=2cm, inner sep=0pt, draw, anchor=south west] {\Huge$-10$};
\draw (18cm,4cm) node [rectangle, minimum size=2cm, inner sep=0pt, draw, anchor=south west] {\Huge$-14$};
\draw (20cm,4cm) node [rectangle, minimum size=2cm, inner sep=0pt, draw, anchor=south west] {\Huge$-18$};
\draw (22cm,4cm) node [rectangle, minimum size=2cm, inner sep=0pt, draw, anchor=south west] {\Huge$-22$};
\draw (24cm,4cm) node [rectangle, minimum size=2cm, inner sep=0pt, draw, anchor=south west] {\Huge$-26$};
\draw (0cm,6cm) node [rectangle, minimum size=2cm, inner sep=0pt, draw, anchor=south west] {\Huge$35$};
\draw (2cm,6cm) node [rectangle, minimum size=2cm, inner sep=0pt, draw, anchor=south west] {\Huge$31$};
\draw (4cm,6cm) node [rectangle, minimum size=2cm, inner sep=0pt, draw, anchor=south west] {\Huge$27$};
\draw (6cm,6cm) node [rectangle, minimum size=2cm, inner sep=0pt, draw, anchor=south west] {\Huge$23$};
\draw (8cm,6cm) node [rectangle, minimum size=2cm, inner sep=0pt, draw, anchor=south west] {\Huge$19$};
\draw (10cm,6cm) node [rectangle, minimum size=2cm, inner sep=0pt, draw, anchor=south west] {\Huge$15$};
\draw (12cm,6cm) node [rectangle, minimum size=2cm, inner sep=0pt, draw, anchor=south west] {\Huge$11$};
\draw (14cm,6cm) node [rectangle, minimum size=2cm, inner sep=0pt, draw, anchor=south west] {\Huge$7$};
\draw (16cm,6cm) node [rectangle, minimum size=2cm, inner sep=0pt, draw, anchor=south west] {\Huge$3$};
\draw (18cm,6cm) node [rectangle, minimum size=2cm, inner sep=0pt, draw, anchor=south west] {\Huge$-1$};
\draw (20cm,6cm) node [rectangle, minimum size=2cm, inner sep=0pt, draw, anchor=south west] {\Huge$-5$};
\draw (22cm,6cm) node [rectangle, minimum size=2cm, inner sep=0pt, draw, anchor=south west] {\Huge$-9$};
\draw (24cm,6cm) node [rectangle, minimum size=2cm, inner sep=0pt, draw, anchor=south west] {\Huge$-13$};
\filldraw[black] (0,0) circle (0.25cm);
\filldraw[black] (26,8) circle (0.25cm);
\draw[dashed](0,0)--(26,8);
\draw[line width=5pt](0,0)--(0,4)--(4,4)--(4,6)--(16,6)--(16,8)--(26,8);
\end{tikzpicture}}
\end{center}
\caption{The $(4,13)$-core partition $\lam=(7,4,4,2,2,1,1,1)$ corresponds to a $4$-flush and $13$-flush abacus $a$ and a lattice path from $(0,0)$ to $(13,4)$ staying above $y=\frac{4}{13}x$.  The boxes to the right of the lattice path correspond to the beads in the abacus.}
\label{fig:Anderson}
\end{figure}

\begin{proposition}\label{p:FMS}
There exists a bijection:
\[
\calL:\left\{\begin{tabular}{c}\textup{antisymmetric $a$-flush and }\\\textup{$b$-flush abacus diagrams}\end{tabular}\right\}
\longleftrightarrow 
\left\{\begin{tabular}{c}\textup{$N$-$E$ lattice paths}\\\textup{$(0,0)\row(\big\lfloor\frac{b}{2}\big\rfloor,\big\lfloor\frac{a}{2}\big\rfloor)$}\end{tabular}\right\}.
\]
\end{proposition}
\begin{proof}
That an antisymmetric abacus is $a$-flush and $b$-flush implies certain conditions on the possible sets of beads and gaps.  For one, when $a$ (or $b$) is odd, then there must be a bead in position $\frac{1-a}{2}$ (or $\frac{1-b}{2}$) and a gap in position $\frac{a+1}{2}$ (or $\frac{b+1}{2}$) since they are an antisymmetric pair and in the same runner.  

When $a$ and $b$ are of opposite parity, then there must be a bead in position $\frac{1-b-a}{2}$ and a gap in position $\frac{1+b+a}{2}$ because they are an antisymmetric pair and the inverse assignment would create an impossibility for position $\frac{1+b-a}{2}$ in terms of being both $a$-flush and $b$-flush.

Organize the integers into a square lattice depending on the parity of $a$.
When $a$ is even, place the integer $\frac{1+b-a}{2}-ai+bj$ in the box with corners $(i,j)$ and $(i+1,j+1)$.  When $a$ is odd, instead insert $\frac{1+2b-a}{2}-ai+bj$.  (See Figure~\ref{fig:FMS}.) 
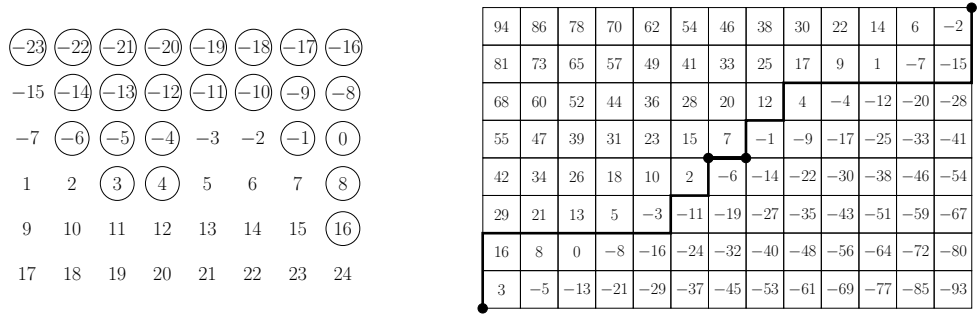
\begin{figure}
\begin{center}

\raisebox{0.3cm}{\scalebox{0.3}{\begin{tikzpicture}
\draw (0cm,4cm) node [circle, minimum size=1.5cm, inner sep=0pt, draw, anchor=south west] {\Huge $-23$};
\draw (2cm,4cm) node [circle, minimum size=1.5cm, inner sep=0pt, draw, anchor=south west] {\Huge $-22$};
\draw (4cm,4cm) node [circle, minimum size=1.5cm, inner sep=0pt, draw, anchor=south west] {\Huge $-21$};
\draw (6cm,4cm) node [circle, minimum size=1.5cm, inner sep=0pt, draw, anchor=south west] {\Huge $-20$};
\draw (8cm,4cm) node [circle, minimum size=1.5cm, inner sep=0pt, draw, anchor=south west] {\Huge $-19$};
\draw (10cm,4cm) node [circle, minimum size=1.5cm, inner sep=0pt, draw, anchor=south west] {\Huge $-18$};
\draw (12cm,4cm) node [circle, minimum size=1.5cm, inner sep=0pt, draw, anchor=south west] {\Huge $-17$};
\draw (14cm,4cm) node [circle, minimum size=1.5cm, inner sep=0pt, draw, anchor=south west] {\Huge $-16$};
\draw (0cm,2cm) node [circle, minimum size=1.5cm, inner sep=0pt, draw=none, anchor=south west] {\Huge $-15$};
\draw (2cm,2cm) node [circle, minimum size=1.5cm, inner sep=0pt, draw, anchor=south west] {\Huge $-14$};
\draw (4cm,2cm) node [circle, minimum size=1.5cm, inner sep=0pt, draw, anchor=south west] {\Huge $-13$};
\draw (6cm,2cm) node [circle, minimum size=1.5cm, inner sep=0pt, draw, anchor=south west] {\Huge $-12$};
\draw (8cm,2cm) node [circle, minimum size=1.5cm, inner sep=0pt, draw, anchor=south west] {\Huge $-11$};
\draw (10cm,2cm) node [circle, minimum size=1.5cm, inner sep=0pt, draw, anchor=south west] {\Huge $-10$};
\draw (12cm,2cm) node [circle, minimum size=1.5cm, inner sep=0pt, draw, anchor=south west] {\Huge $-9$};
\draw (14cm,2cm) node [circle, minimum size=1.5cm, inner sep=0pt, draw, anchor=south west] {\Huge $-8$};
\draw (0cm,0cm) node [circle, minimum size=1.5cm, inner sep=0pt, draw=none, anchor=south west] {\Huge $-7$};
\draw (2cm,0cm) node [circle, minimum size=1.5cm, inner sep=0pt, draw, anchor=south west] {\Huge $-6$};
\draw (4cm,0cm) node [circle, minimum size=1.5cm, inner sep=0pt, draw, anchor=south west] {\Huge $-5$};
\draw (6cm,0cm) node [circle, minimum size=1.5cm, inner sep=0pt, draw, anchor=south west] {\Huge $-4$};
\draw (8cm,0cm) node [circle, minimum size=1.5cm, inner sep=0pt, draw=none, anchor=south west] {\Huge $-3$};
\draw (10cm,0cm) node [circle, minimum size=1.5cm, inner sep=0pt, draw=none, anchor=south west] {\Huge $-2$};
\draw (12cm,0cm) node [circle, minimum size=1.5cm, inner sep=0pt, draw, anchor=south west] {\Huge $-1$};
\draw (14cm,0cm) node [circle, minimum size=1.5cm, inner sep=0pt, draw, anchor=south west] {\Huge $0$};
\draw (0cm,-2cm) node [circle, minimum size=1.5cm, inner sep=0pt, draw=none, anchor=south west] {\Huge $1$};
\draw (2cm,-2cm) node [circle, minimum size=1.5cm, inner sep=0pt, draw=none, anchor=south west] {\Huge $2$};
\draw (4cm,-2cm) node [circle, minimum size=1.5cm, inner sep=0pt, draw, anchor=south west] {\Huge $3$};
\draw (6cm,-2cm) node [circle, minimum size=1.5cm, inner sep=0pt, draw, anchor=south west] {\Huge $4$};
\draw (8cm,-2cm) node [circle, minimum size=1.5cm, inner sep=0pt, draw=none, anchor=south west] {\Huge $5$};
\draw (10cm,-2cm) node [circle, minimum size=1.5cm, inner sep=0pt, draw=none, anchor=south west] {\Huge $6$};
\draw (12cm,-2cm) node [circle, minimum size=1.5cm, inner sep=0pt, draw=none, anchor=south west] {\Huge $7$};
\draw (14cm,-2cm) node [circle, minimum size=1.5cm, inner sep=0pt, draw, anchor=south west] {\Huge $8$};
\draw (0cm,-4cm) node [circle, minimum size=1.5cm, inner sep=0pt, draw=none, anchor=south west] {\Huge $9$};
\draw (2cm,-4cm) node [circle, minimum size=1.5cm, inner sep=0pt, draw=none, anchor=south west] {\Huge $10$};
\draw (4cm,-4cm) node [circle, minimum size=1.5cm, inner sep=0pt, draw=none, anchor=south west] {\Huge $11$};
\draw (6cm,-4cm) node [circle, minimum size=1.5cm, inner sep=0pt, draw=none, anchor=south west] {\Huge $12$};
\draw (8cm,-4cm) node [circle, minimum size=1.5cm, inner sep=0pt, draw=none, anchor=south west] {\Huge $13$};
\draw (10cm,-4cm) node [circle, minimum size=1.5cm, inner sep=0pt, draw=none, anchor=south west] {\Huge $14$};
\draw (12cm,-4cm) node [circle, minimum size=1.5cm, inner sep=0pt, draw=none, anchor=south west] {\Huge $15$};
\draw (14cm,-4cm) node [circle, minimum size=1.5cm, inner sep=0pt, draw, anchor=south west] {\Huge $16$};
\draw (0cm,-6cm) node [circle, minimum size=1.5cm, inner sep=0pt, draw=none, anchor=south west] {\Huge $17$};
\draw (2cm,-6cm) node [circle, minimum size=1.5cm, inner sep=0pt, draw=none, anchor=south west] {\Huge $18$};
\draw (4cm,-6cm) node [circle, minimum size=1.5cm, inner sep=0pt, draw=none, anchor=south west] {\Huge $19$};
\draw (6cm,-6cm) node [circle, minimum size=1.5cm, inner sep=0pt, draw=none, anchor=south west] {\Huge $20$};
\draw (8cm,-6cm) node [circle, minimum size=1.5cm, inner sep=0pt, draw=none, anchor=south west] {\Huge $21$};
\draw (10cm,-6cm) node [circle, minimum size=1.5cm, inner sep=0pt, draw=none, anchor=south west] {\Huge $22$};
\draw (12cm,-6cm) node [circle, minimum size=1.5cm, inner sep=0pt, draw=none, anchor=south west] {\Huge $23$};
\draw (14cm,-6cm) node [circle, minimum size=1.5cm, inner sep=0pt, draw=none, anchor=south west] {\Huge $24$};
\end{tikzpicture}}}\qquad\qquad
\scalebox{0.25}{\begin{tikzpicture}
\draw (0cm,0cm) node [rectangle, minimum size=2cm, inner sep=0pt, draw, anchor=south west] {\Huge$3$};
\draw (2cm,0cm) node [rectangle, minimum size=2cm, inner sep=0pt, draw, anchor=south west] {\Huge$-5$};
\draw (4cm,0cm) node [rectangle, minimum size=2cm, inner sep=0pt, draw, anchor=south west] {\Huge$-13$};
\draw (6cm,0cm) node [rectangle, minimum size=2cm, inner sep=0pt, draw, anchor=south west] {\Huge$-21$};
\draw (8cm,0cm) node [rectangle, minimum size=2cm, inner sep=0pt, draw, anchor=south west] {\Huge$-29$};
\draw (10cm,0cm) node [rectangle, minimum size=2cm, inner sep=0pt, draw, anchor=south west] {\Huge$-37$};
\draw (12cm,0cm) node [rectangle, minimum size=2cm, inner sep=0pt, draw, anchor=south west] {\Huge$-45$};
\draw (14cm,0cm) node [rectangle, minimum size=2cm, inner sep=0pt, draw, anchor=south west] {\Huge$-53$};
\draw (16cm,0cm) node [rectangle, minimum size=2cm, inner sep=0pt, draw, anchor=south west] {\Huge$-61$};
\draw (18cm,0cm) node [rectangle, minimum size=2cm, inner sep=0pt, draw, anchor=south west] {\Huge$-69$};
\draw (20cm,0cm) node [rectangle, minimum size=2cm, inner sep=0pt, draw, anchor=south west] {\Huge$-77$};
\draw (22cm,0cm) node [rectangle, minimum size=2cm, inner sep=0pt, draw, anchor=south west] {\Huge$-85$};
\draw (24cm,0cm) node [rectangle, minimum size=2cm, inner sep=0pt, draw, anchor=south west] {\Huge$-93$};
\draw (0cm,2cm) node [rectangle, minimum size=2cm, inner sep=0pt, draw, anchor=south west] {\Huge$16$};
\draw (2cm,2cm) node [rectangle, minimum size=2cm, inner sep=0pt, draw, anchor=south west] {\Huge$8$};
\draw (4cm,2cm) node [rectangle, minimum size=2cm, inner sep=0pt, draw, anchor=south west] {\Huge$0$};
\draw (6cm,2cm) node [rectangle, minimum size=2cm, inner sep=0pt, draw, anchor=south west] {\Huge$-8$};
\draw (8cm,2cm) node [rectangle, minimum size=2cm, inner sep=0pt, draw, anchor=south west] {\Huge$-16$};
\draw (10cm,2cm) node [rectangle, minimum size=2cm, inner sep=0pt, draw, anchor=south west] {\Huge$-24$};
\draw (12cm,2cm) node [rectangle, minimum size=2cm, inner sep=0pt, draw, anchor=south west] {\Huge$-32$};
\draw (14cm,2cm) node [rectangle, minimum size=2cm, inner sep=0pt, draw, anchor=south west] {\Huge$-40$};
\draw (16cm,2cm) node [rectangle, minimum size=2cm, inner sep=0pt, draw, anchor=south west] {\Huge$-48$};
\draw (18cm,2cm) node [rectangle, minimum size=2cm, inner sep=0pt, draw, anchor=south west] {\Huge$-56$};
\draw (20cm,2cm) node [rectangle, minimum size=2cm, inner sep=0pt, draw, anchor=south west] {\Huge$-64$};
\draw (22cm,2cm) node [rectangle, minimum size=2cm, inner sep=0pt, draw, anchor=south west] {\Huge$-72$};
\draw (24cm,2cm) node [rectangle, minimum size=2cm, inner sep=0pt, draw, anchor=south west] {\Huge$-80$};
\draw (0cm,4cm) node [rectangle, minimum size=2cm, inner sep=0pt, draw, anchor=south west] {\Huge$29$};
\draw (2cm,4cm) node [rectangle, minimum size=2cm, inner sep=0pt, draw, anchor=south west] {\Huge$21$};
\draw (4cm,4cm) node [rectangle, minimum size=2cm, inner sep=0pt, draw, anchor=south west] {\Huge$13$};
\draw (6cm,4cm) node [rectangle, minimum size=2cm, inner sep=0pt, draw, anchor=south west] {\Huge$5$};
\draw (8cm,4cm) node [rectangle, minimum size=2cm, inner sep=0pt, draw, anchor=south west] {\Huge$-3$};
\draw (10cm,4cm) node [rectangle, minimum size=2cm, inner sep=0pt, draw, anchor=south west] {\Huge$-11$};
\draw (12cm,4cm) node [rectangle, minimum size=2cm, inner sep=0pt, draw, anchor=south west] {\Huge$-19$};
\draw (14cm,4cm) node [rectangle, minimum size=2cm, inner sep=0pt, draw, anchor=south west] {\Huge$-27$};
\draw (16cm,4cm) node [rectangle, minimum size=2cm, inner sep=0pt, draw, anchor=south west] {\Huge$-35$};
\draw (18cm,4cm) node [rectangle, minimum size=2cm, inner sep=0pt, draw, anchor=south west] {\Huge$-43$};
\draw (20cm,4cm) node [rectangle, minimum size=2cm, inner sep=0pt, draw, anchor=south west] {\Huge$-51$};
\draw (22cm,4cm) node [rectangle, minimum size=2cm, inner sep=0pt, draw, anchor=south west] {\Huge$-59$};
\draw (24cm,4cm) node [rectangle, minimum size=2cm, inner sep=0pt, draw, anchor=south west] {\Huge$-67$};
\draw (0cm,6cm) node [rectangle, minimum size=2cm, inner sep=0pt, draw, anchor=south west] {\Huge$42$};
\draw (2cm,6cm) node [rectangle, minimum size=2cm, inner sep=0pt, draw, anchor=south west] {\Huge$34$};
\draw (4cm,6cm) node [rectangle, minimum size=2cm, inner sep=0pt, draw, anchor=south west] {\Huge$26$};
\draw (6cm,6cm) node [rectangle, minimum size=2cm, inner sep=0pt, draw, anchor=south west] {\Huge$18$};
\draw (8cm,6cm) node [rectangle, minimum size=2cm, inner sep=0pt, draw, anchor=south west] {\Huge$10$};
\draw (10cm,6cm) node [rectangle, minimum size=2cm, inner sep=0pt, draw, anchor=south west] {\Huge$2$};
\draw (12cm,6cm) node [rectangle, minimum size=2cm, inner sep=0pt, draw, anchor=south west] {\Huge$-6$};
\draw (14cm,6cm) node [rectangle, minimum size=2cm, inner sep=0pt, draw, anchor=south west] {\Huge$-14$};
\draw (16cm,6cm) node [rectangle, minimum size=2cm, inner sep=0pt, draw, anchor=south west] {\Huge$-22$};
\draw (18cm,6cm) node [rectangle, minimum size=2cm, inner sep=0pt, draw, anchor=south west] {\Huge$-30$};
\draw (20cm,6cm) node [rectangle, minimum size=2cm, inner sep=0pt, draw, anchor=south west] {\Huge$-38$};
\draw (22cm,6cm) node [rectangle, minimum size=2cm, inner sep=0pt, draw, anchor=south west] {\Huge$-46$};
\draw (24cm,6cm) node [rectangle, minimum size=2cm, inner sep=0pt, draw, anchor=south west] {\Huge$-54$};
\draw (0cm,8cm) node [rectangle, minimum size=2cm, inner sep=0pt, draw, anchor=south west] {\Huge$55$};
\draw (2cm,8cm) node [rectangle, minimum size=2cm, inner sep=0pt, draw, anchor=south west] {\Huge$47$};
\draw (4cm,8cm) node [rectangle, minimum size=2cm, inner sep=0pt, draw, anchor=south west] {\Huge$39$};
\draw (6cm,8cm) node [rectangle, minimum size=2cm, inner sep=0pt, draw, anchor=south west] {\Huge$31$};
\draw (8cm,8cm) node [rectangle, minimum size=2cm, inner sep=0pt, draw, anchor=south west] {\Huge$23$};
\draw (10cm,8cm) node [rectangle, minimum size=2cm, inner sep=0pt, draw, anchor=south west] {\Huge$15$};
\draw (12cm,8cm) node [rectangle, minimum size=2cm, inner sep=0pt, draw, anchor=south west] {\Huge$7$};
\draw (14cm,8cm) node [rectangle, minimum size=2cm, inner sep=0pt, draw, anchor=south west] {\Huge$-1$};
\draw (16cm,8cm) node [rectangle, minimum size=2cm, inner sep=0pt, draw, anchor=south west] {\Huge$-9$};
\draw (18cm,8cm) node [rectangle, minimum size=2cm, inner sep=0pt, draw, anchor=south west] {\Huge$-17$};
\draw (20cm,8cm) node [rectangle, minimum size=2cm, inner sep=0pt, draw, anchor=south west] {\Huge$-25$};
\draw (22cm,8cm) node [rectangle, minimum size=2cm, inner sep=0pt, draw, anchor=south west] {\Huge$-33$};
\draw (24cm,8cm) node [rectangle, minimum size=2cm, inner sep=0pt, draw, anchor=south west] {\Huge$-41$};
\draw (0cm,10cm) node [rectangle, minimum size=2cm, inner sep=0pt, draw, anchor=south west] {\Huge$68$};
\draw (2cm,10cm) node [rectangle, minimum size=2cm, inner sep=0pt, draw, anchor=south west] {\Huge$60$};
\draw (4cm,10cm) node [rectangle, minimum size=2cm, inner sep=0pt, draw, anchor=south west] {\Huge$52$};
\draw (6cm,10cm) node [rectangle, minimum size=2cm, inner sep=0pt, draw, anchor=south west] {\Huge$44$};
\draw (8cm,10cm) node [rectangle, minimum size=2cm, inner sep=0pt, draw, anchor=south west] {\Huge$36$};
\draw (10cm,10cm) node [rectangle, minimum size=2cm, inner sep=0pt, draw, anchor=south west] {\Huge$28$};
\draw (12cm,10cm) node [rectangle, minimum size=2cm, inner sep=0pt, draw, anchor=south west] {\Huge$20$};
\draw (14cm,10cm) node [rectangle, minimum size=2cm, inner sep=0pt, draw, anchor=south west] {\Huge$12$};
\draw (16cm,10cm) node [rectangle, minimum size=2cm, inner sep=0pt, draw, anchor=south west] {\Huge$4$};
\draw (18cm,10cm) node [rectangle, minimum size=2cm, inner sep=0pt, draw, anchor=south west] {\Huge$-4$};
\draw (20cm,10cm) node [rectangle, minimum size=2cm, inner sep=0pt, draw, anchor=south west] {\Huge$-12$};
\draw (22cm,10cm) node [rectangle, minimum size=2cm, inner sep=0pt, draw, anchor=south west] {\Huge$-20$};
\draw (24cm,10cm) node [rectangle, minimum size=2cm, inner sep=0pt, draw, anchor=south west] {\Huge$-28$};
\draw (0cm,12cm) node [rectangle, minimum size=2cm, inner sep=0pt, draw, anchor=south west] {\Huge$81$};
\draw (2cm,12cm) node [rectangle, minimum size=2cm, inner sep=0pt, draw, anchor=south west] {\Huge$73$};
\draw (4cm,12cm) node [rectangle, minimum size=2cm, inner sep=0pt, draw, anchor=south west] {\Huge$65$};
\draw (6cm,12cm) node [rectangle, minimum size=2cm, inner sep=0pt, draw, anchor=south west] {\Huge$57$};
\draw (8cm,12cm) node [rectangle, minimum size=2cm, inner sep=0pt, draw, anchor=south west] {\Huge$49$};
\draw (10cm,12cm) node [rectangle, minimum size=2cm, inner sep=0pt, draw, anchor=south west] {\Huge$41$};
\draw (12cm,12cm) node [rectangle, minimum size=2cm, inner sep=0pt, draw, anchor=south west] {\Huge$33$};
\draw (14cm,12cm) node [rectangle, minimum size=2cm, inner sep=0pt, draw, anchor=south west] {\Huge$25$};
\draw (16cm,12cm) node [rectangle, minimum size=2cm, inner sep=0pt, draw, anchor=south west] {\Huge$17$};
\draw (18cm,12cm) node [rectangle, minimum size=2cm, inner sep=0pt, draw, anchor=south west] {\Huge$9$};
\draw (20cm,12cm) node [rectangle, minimum size=2cm, inner sep=0pt, draw, anchor=south west] {\Huge$1$};
\draw (22cm,12cm) node [rectangle, minimum size=2cm, inner sep=0pt, draw, anchor=south west] {\Huge$-7$};
\draw (24cm,12cm) node [rectangle, minimum size=2cm, inner sep=0pt, draw, anchor=south west] {\Huge$-15$};
\draw (0cm,14cm) node [rectangle, minimum size=2cm, inner sep=0pt, draw, anchor=south west] {\Huge$94$};
\draw (2cm,14cm) node [rectangle, minimum size=2cm, inner sep=0pt, draw, anchor=south west] {\Huge$86$};
\draw (4cm,14cm) node [rectangle, minimum size=2cm, inner sep=0pt, draw, anchor=south west] {\Huge$78$};
\draw (6cm,14cm) node [rectangle, minimum size=2cm, inner sep=0pt, draw, anchor=south west] {\Huge$70$};
\draw (8cm,14cm) node [rectangle, minimum size=2cm, inner sep=0pt, draw, anchor=south west] {\Huge$62$};
\draw (10cm,14cm) node [rectangle, minimum size=2cm, inner sep=0pt, draw, anchor=south west] {\Huge$54$};
\draw (12cm,14cm) node [rectangle, minimum size=2cm, inner sep=0pt, draw, anchor=south west] {\Huge$46$};
\draw (14cm,14cm) node [rectangle, minimum size=2cm, inner sep=0pt, draw, anchor=south west] {\Huge$38$};
\draw (16cm,14cm) node [rectangle, minimum size=2cm, inner sep=0pt, draw, anchor=south west] {\Huge$30$};
\draw (18cm,14cm) node [rectangle, minimum size=2cm, inner sep=0pt, draw, anchor=south west] {\Huge$22$};
\draw (20cm,14cm) node [rectangle, minimum size=2cm, inner sep=0pt, draw, anchor=south west] {\Huge$14$};
\draw (22cm,14cm) node [rectangle, minimum size=2cm, inner sep=0pt, draw, anchor=south west] {\Huge$6$};
\draw (24cm,14cm) node [rectangle, minimum size=2cm, inner sep=0pt, draw, anchor=south west] {\Huge$-2$};
\draw[line width=4pt](0,0)--(0,4)--(10,4)--(10,6)--(12,6)--(12,8)--(14,8)--(14,10)--(16,10)--(16,12)--(26,12)--(26,16);
\draw[line width=6pt](12,8)--(14,8);
\filldraw[black] (0,0) circle (0.25cm);
\filldraw[black] (12,8) circle (0.25cm);
\filldraw[black] (14,8) circle (0.25cm);
\filldraw[black] (26,16) circle (0.25cm);
\end{tikzpicture}}
\end{center}
\caption{The placement of integers in boxes $(i,j)$ for $8$-flush and $13$-flush abaci for $0\leq i\leq 7$, $0\leq j\leq 12$.  The flush conditions force $-6$ to be a bead and $7$ to be a gap. Because of the symmetry, a lattice path is completely defined by the portion from $(0,0)$ to $(4,6)$.  An example of an antisymmetric abacus and its corresponding lattice path is shown.}
\label{fig:FMS}
\end{figure}

As in type~$A$ case, the $a$-flush condition can be read horizontally and the $b$-flush condition can be read vertically.  Given 
a type~$C$ abacus $A$ that is both $a$-flush and $b$-flush, the dividing lattice path between integers that are beads of $A$ and integers that are gaps of $A$ is a lattice path $\calL(A)$ from $(0,0)$ to $(b,a)$ that has rotational symmetry about the point $\big(\frac{b}{2},\frac{a}{2}\big)$.  
% that separates the gaps (to the left and above) from the beads (to the right and below).  
The inherent symmetry implies that we need only consider the (unrestricted) lattice path from $(0,0)$ to $\big(\big\lfloor\frac{b}{2}\big\rfloor,\big\lfloor\frac{a}{2}\big\rfloor\big)$.
\end{proof}

\begin{remark}
The diagonal hook lengths discussed by Ford, Mai, and Sze can be recovered by analyzing the set of positive beads.  An antisymmetric bead-gap pair for a positive bead $x$ corresponds to a diagonal hook of length $2x-1$.  Indeed, we recover the numbers in the lattice of Ford, Mai, and Sze after matching their indexing conventions and applying the transformation $f(x)=2x-1$ to the numbers in our lattice.
\end{remark}

\begin{proof}[Proof of Theorem~\ref{t:qA}]

In type~$A$, the bijection of Anderson \cite[Proposition~1]{Anderson} specializes to a bijection 
\[
\calL:\left\{\begin{tabular}{c}\textup{$n$-flush and $(mn+1)$-flush}\\\textup{abacus diagrams}\end{tabular}\right\}
\longleftrightarrow 
\left\{\begin{tabular}{c}\textup{$N$-$E$ lattice paths}\\\textup{$(0,0)\row(n,n)$}\\\textup{on or above $y=x$}\end{tabular}\right\},
\]
where the abacus diagram is normalized and drawn on runners $0$ through $n-1$.
In type~$C$, the bijection in Proposition~\ref{p:FMS} restricts to the bijection
\[
\calL:\left\{\begin{tabular}{c}\textup{antisymmetric $2n$-flush and}\\\textup{$(2n+1)$-flush abacus diagrams}\end{tabular}\right\}
\longleftrightarrow 
\left\{\begin{tabular}{c}\textup{$N$-$E$ lattice paths}\\\textup{$(0,0)\row(n,n)$}\end{tabular}\right\}.
\]
Let $L$ be a lattice path and let $\maj(L)$ be the standard major index statistic defined earlier. We will show that $\maj_A(\lam)=\maj(L)$ and $2\hspace{.02in}\maj_C(\lam)=\maj(L)$.  The desired result follows from a classical result of MacMahon; see for example \cite[Chapter 1]{haglund}.

We must determine the positions of East steps followed by North steps in $L$.
Since the position of the lowest bead in runner $i$ corresponds to the $i$-th
North step, an East step before the $i$-th North step occurs if the level of the
lowest bead in runner $i$ is less than or equal to the lowest bead in runner
$i-1$, which is exactly the condition that the sequence $x$ has a descent in position $i$.  

Now we must determine the step along $L$ where this North step occurs to see
what its contribution to $\maj_A(\lam)$ or $\maj_C(\lam)$ should be.  A North step always
corresponds to changing runners in the abacus.  An East step corresponds to
walking up the levels in the runner.  So if the weak descent of $x$ occurs in
position $i$ with a bead on level $x_i$, then this corresponds to having
traversed $i$ North steps (in type~$C$, $(i-1)$ North steps) and $(i-x_i)$
East steps, which contributes $2i-x_i$ to $\maj_A(\lam)$ (in type~$C$, it contributes $2(2i-x_i-1)$ to $\maj_C(\lam)$). 
The sum over all descents gives Equations~\eqref{eq:majA} and \eqref{eq:majC}.
\end{proof}

%%%%%%%%%%%%%%%%%%%%%%%%%%%%%%%%%%%%%%%%%%%%%%%%%%%%%%%%%%%%%%%%%%%%%
%  Section
%%%%%%%%%%%%%%%%%%%%%%%%%%%%%%%%%%%%%%%%%%%%%%%%%%%%%%%%%%%%%%%%%%%%%
\section{Further questions}
\label{sec:fq}

In this section, we present some directions for future research.

\begin{question}
The dominant regions of the Shi arrangement form a set of representatives for
certain orbits inside the set of all (not necessarily dominant) Shi regions.
For example, in type~$A$ one can label the regions of the Shi arrangement by a
Dyck path (representing some dominant Shi region) together with a permutation
that is a minimal length coset representative for a quotient that is defined by
the choice of Dyck path.

If we instead use a simultaneous core to represent the dominant Shi region in
type~$A$ or type~$C$, what additional data would we need to add as a decoration
in order to parameterize the full set of (not necessarily dominant) Shi regions?
Recent work of M\'esz\'aros \cite{meszaros} contains relevant combinatorics for
type $C$.
\end{question}

\begin{question}
From the perspective of abacus diagrams, the simultaneous cores we have studied
are defined entirely in terms of conditions that have the form ``If a bead
exists at position $i$ then a bead exists at position $f(i)$,'' where $f(i)$ is
the function $i-j$ with $j$ constant.  As we have seen, these ``convexity''
conditions conspire to produce a finite set of abacus diagrams when $j$ is
relatively prime to the number of runners $R$.

It is natural to consider more general types of functions $f(i)$.  For example,
when defining abacus diagrams that correspond to dominant regions of the 
type $B$ or type $D$ Shi arrangements, we must impose
distinct flush conditions depending on the column containing $i$ in the abacus.

Which functions $f$ produce finite sets of abaci?  Is it possible to enumerate
these sets directly from the abacus diagram and the conditions imposed by $f$?
Are there other natural classes of partitions that are defined in terms of
convex conditions on abaci?
\end{question}

\begin{question} 
Our argument that the dominant regions of the $m$-Shi arrangement correspond to
simultaneous core partitions does not generalize to types $B$ and
$D$.  The condition analogous to Lemma~\ref{l:right_descent} for
$s_i$ to be a right descent on the abacus involves non-adjacent columns in
these types.

Starting from the fact that that the set of dominant alcoves in these types
correspond to even $(2n)$-core partitions, it would be natural to look for a
simple criterion on these partitions that selects the subset of minimal or
bounded $m$-Shi alcoves.  The conditions of being $(2mn \pm 1)$-core do not
produce the correct subsets.
\end{question}

%%%%%%%%%%%%%%%%%%%%%%%%%%%%%%%%%%%%%%%%%%%%%%%%%%%%%%%%%%%%%%%%%%%%%
%  Acknowledgements
%%%%%%%%%%%%%%%%%%%%%%%%%%%%%%%%%%%%%%%%%%%%%%%%%%%%%%%%%%%%%%%%%%%%%
\section*{Acknowledgements}

We would like to thank Monica Vazirani for discussions of related topics and the Institute for Computational and Experimental Research in Mathematics for supporting stimulating collaboration.  We thank the referees for suggestions that have improved the clarity of the exposition.

 D.\ Armstrong was partially supported by NSF award DMS-1001825. C.\ R.\ H.\ Hanusa gratefully acknowledges support from PSC-CUNY Research Awards TRADA-43-127 and TRADA-44-168.   B.\ C.\ Jones was supported in part by a CSM Summer Research Grant from James Madison University.

%%%%%%%%%%%%%%%%%%%%%%%%%%%%%%%%%%%%%%%%%%%%%%%%%%%%%%%%%%%%%%%
%\section*{Bibliography}

\end{document}